\documentclass[12pt,english]{article}
\usepackage[T1]{fontenc}
\usepackage[latin9]{inputenc}
\usepackage{mathrsfs}
\usepackage{amscd}
\usepackage{amsmath}
\usepackage{amsthm}
\usepackage{amssymb}
\usepackage{setspace}
\makeatletter
\numberwithin{equation}{section}

\theoremstyle{remark}
\newtheorem*{acknowledgement*}{\protect\acknowledgementname}
\theoremstyle{plain}
\newtheorem{thm}{\protect\theoremname}[section]
\theoremstyle{plain}
\newtheorem{lem}[thm]{\protect\lemmaname}
\theoremstyle{plain}
\newtheorem{prop}[thm]{\protect\propositionname}
\theoremstyle{remark}
\newtheorem{rem}[thm]{\protect\remarkname}

\makeatother

\usepackage{babel}
\providecommand{\acknowledgementname}{Acknowledgement}
\providecommand{\lemmaname}{Lemma}
\providecommand{\propositionname}{Proposition}
\providecommand{\remarkname}{Remark}
\providecommand{\theoremname}{Theorem}
\date{}

\begin{document}
\title{Gelfand hypergeometric function as a solution to the 2-dimensional
Toda-Hirota equation}
\author{Hironobu Kimura,\\
 Department of Mathematics, Graduate School of Science and\\
  Technology, Kumamoto University}
\maketitle

\begin{abstract}
We construct solutions of the $2$-dimensional Toda-Hirota equation (2dTHE)
expressed by the solutions of the system of so-called Euler-Poisson-Darboux equations
(EPD) in $N$ complex variables. The system of EPD arises naturally from the differential equations
which form a main body of the system characterizing the Gelfand hypergeometric
function (Gelfand HGF) on the Grassmannian $\mathrm{GM}(2,N)$. Using this
link and the contiguity relations for the Gelfand HGF, which are constructed
from root vectors for the root $\epsilon_{i}-\epsilon_{j}$ for $\mathfrak{gl}(N)$, we
show that the Gelfand HGF gives solutions of the  2dTHE. 
\end{abstract}

\section{Introduction }

The purpose of this paper is to make clear how the Gelfand hypergeometric
function (Gelfand HGF) on the complex Grassmannian manifold is regarded
as a solution of the 2-dimensional Toda-Hirota equation (2dTHE):
\begin{equation}
\partial_{x}\partial_{y}\log\tau_{n}=\frac{\tau_{n+1}\tau_{n-1}}{\tau_{n}^{2}},\quad n\in\mathbb{Z},\label{eq:intro-1}
\end{equation}
which is a bilinear form of the 2-dimensional Toda equation and is
an extension of 
\begin{equation}
\frac{d^{2}}{dx^{2}}\log\tau_{n}=\frac{\tau_{n+1}\tau_{n-1}}{\tau_{n}^{2}},\quad n\in\mathbb{Z}.\label{eq:intro-2}
\end{equation}
The equation (\ref{eq:intro-1}) or (\ref{eq:intro-2}) is one of
the best known nonlinear integrable systems and its structure of the
solutions are studied from various points of view \cite{Darboux,Hirota,kame-1,Nakamura-A,Okamoto-2,Olshanetsky,Satuma}.

These equations admit various type of solutions, rational solutions,
soliton solutions, for example. We are interested in the solutions
 related to the special functions, for example the Gauss hypergeometric
function (HGF)
\[
\,_{2}F_{1}(a,b,c;x)=\sum_{k=1}^{\infty}\frac{(a)_{k}(b)_{k}}{(c)_{k}k!}x^{k}=C\int_{0}^{1}u^{a-1}(1-u)^{c-a-1}(1-xu)^{-b}du,
\]
and its confluent family: Kummer's confluent HGF, Bessel function,
Hermite-Weber function, and HGFs of several variables, where $(a)_{k}=\Gamma(a+k)/\Gamma(a)$
and $C=\Gamma(c)/\Gamma(a)\Gamma(c-a)$, see \cite{Appell-2,Erdelyi,IKSY}. There
are works on this subject \cite{kame-1,Okamoto-1,Satuma}. In \cite{Okamoto-1},
Okamoto constructed the solutions of (\ref{eq:intro-2}) expressed
in terms of the above Gauss HGF family using the contiguity relations
for them. Moreover, he obtained in \cite{Okamoto-2} the solution of
(\ref{eq:intro-1}) expressed by Appell's HGFs of two variables:
\begin{align*}
F_{1}(x,y) & =C_{1}\int_{0}^{1}u^{\alpha-1}(1-u)^{\gamma-\alpha-1}(1-xu)^{-\beta}(1-yu)^{-\beta'}du,\\
F_{2}(x,y) & =C_{2}\iint u^{\beta-1}v^{\beta'-1}(1-u)^{\gamma-\beta-1}(1-u)^{\gamma'-\beta'-1}(1-xu-yv)^{-\alpha}dudv.
\end{align*}
Similar results are also obtained by Kametaka \cite{kame-1} including
the solutions expressed in terms of confluent type functions of $F_{1}$
and $F_{2}$. Their method is based on the work of G. Darboux \cite{Darboux}
who discussed the mechanism of producing new surfaces in the Euclidean
space $\mathbb{R}^{n}$ successively. The key idea is to consider a simple
hyperbolic operator 
\[
M=\partial_{x}\partial_{y}+a(x,y)\partial_{x}+b(x,y)\partial_{y}+c(x,y),
\]
and to discuss the decomposability of $M$ into the 1st order differential
operators. Write $M$ in the form $M=(\partial_{x}+b)(\partial_{y}+a)-h$, where
$h=h(x,y)$ measures the decomposability and is called the invariant
of $M$. If $h\neq0$, one can construct an operator $M_{+}$ of the
same type by considering the change of unknown $u\mapsto u_{+}=(\partial_{y}+a)u$
for $Mu=0$. Apply the same process to the new operator $M_{+}$ and
so on. Then, starting from $M_{0}=M$, one obtains the sequence of
hyperbolic operators $\{M_{n}\}_{n\geq0}$ with the invariants $h_{n}$,
which is called the Laplace sequence. Surprisingly the invariants
$h_{n}$ satisfies the equation
\[
\partial_{x}\partial_{y}\log h_{n}=-h_{n+1}+2h_{n}-h_{n-1},
\]
which is connected to the 2dTHE by $\partial_{x}\partial_{y}\log\tau_{n}=-h_{n-1}$.
Special type of hyperbolic operator gives a particular solution of the 
2dTHE. Starting from the hyperbolic operator 
\[
M=\partial_{x}\partial_{y}+\frac{\beta'}{x-y}\partial_{x}+\frac{\beta}{y-x}\partial_{y},
\]
called the Euler-Poisson-Darboux operator (EPD operator), one obtains
a simple solution $\{t_{n}\}$ of the 2dTHE from the sequence $\{M_{n}\}_{n\geq0}$.
Then together with the appropriately chosen solution $u_{n}$ of $M_{n}u=0$,
which can be expressed explicitly in terms of $F_{1},$ they obtained
a solution of the 2dTHE in the form $\tau_{n}=t_{n}u_{n}$. The process
$t_{n}\to\tau_{n}$, which gives a new solution $\tau_{n}$ from the
old $t_{n}$, is called the B\"acklund transformation.

The Gelfand HGF on the complex Grassmannian manifold $\mathrm{GM}(r,N)$
is a natural generalization of the HGFs appeared above, and it enables
a unified approach to understand various aspects of classical HGFs
\cite{Gelfand,Kimura-Haraoka,Kimura-H-T}. The Gelfand HGF is defined
as a Radon transform of a character of the maximal abelian subgroup
$H_{\lambda}\subset\mathrm{GL}(N)$, which is specified by a partition $\lambda=(n_{1},\dots,n_{\ell})$
of $N$. When $\lambda=(1,\dots,1)$, $H=H_{\lambda}$ is a Cartan subgroup
and the HGF is said to be of non-confluent type. For example, the
Gauss HGF and its confluent family: Kummer, Bessel, Hermite-Weber,
are understood as the Gelfand HGF on $\mathrm{GM}(2,4)$ corresponding to
the partitions $(1,1,1,1),(2,1,1),(2,2)$ and $(3,1)$, respectively
\cite{Kimura-Haraoka}. Also we see that Appell's $F_{1}$ and $F_{2}$
are the Gelfand HGFs of non-confluent type on $\mathrm{GM}(2,5)$ and on
a certain codimension $2$ stratum of $\mathrm{GM}(3,6)$ \cite{Gelfand}. 

In this paper we consider the Gelfand HGF on $\mathrm{GM}(2,N)$ of non-confluent
type, which is defined on the Zariski open subset of $\mathrm{Mat}(2,N)$:
\[
Z=\{z=(z_{1},\dots,z_{N})\in\mathrm{Mat}(2,N)\mid\det(z_{i},z_{j})\neq0\text{ for any }i\neq j\}
\]
and is given by the 1-dimensional integral
\[
F(z;\alpha)=\int_{C(z)}\prod_{1\leq j\leq N}(z_{1,j}+z_{2,j}u)^{\alpha_{j}}du.
\]
The main result, Theorem \ref{thm:main}, asserts the following. Denote
by $\Phi(x;\alpha)$ the restriction of $F(z;\alpha)$ to the subspace of
$Z$:
\[
X=\left\{ \mathbf{x}=\left(\begin{array}{ccc}
x_{1} & \dots & x_{N}\\
1 & \dots & 1
\end{array}\right)\mid x_{a}\neq x_{b}\;\text{for }\forall a\neq b\right\} \subset Z.
\]
For any $1\leq i\neq j\le N$, put
\[
\tau_{n}(x)=\frac{\Gamma(\alpha_{j}+1)}{\Gamma(\alpha_{j}-n+1)}B(\alpha_{i},\alpha_{j};n)\cdot(x_{i}-x_{j})^{(\alpha_{i}+n)(\alpha_{j}-n)}\Phi(x;\alpha+n(e_{i}-e_{j})),
\]
where $\alpha+n(e_{i}-e_{j})=(\alpha_{1},\dots,\alpha_{i}+n,\dots,\alpha_{j}-n,\dots,\alpha_{N})$,
and  $B(\alpha_{i},\alpha_{j};n)=A^{n}\prod_{k=0}^{n-1}\left(\prod_{l=1}^{k}(\alpha_{i}+l)(\alpha_{j}-l+1)\right) $ in the case $n\geq0$. Then $\tau_{n}$ gives a solution to the 2dTHE
\[
\partial_{i}\partial_{j}\log\tau_{n}=\frac{\tau_{n+1}\tau_{n-1}}{\tau_{n}^{2}},\quad n\in\mathbb{Z},
\]
where $\partial_{i}=\partial/\partial x_{i}$. In constructing the solution, the
contiguity relations for $F(z;\alpha)$ plays an essential role. As for
the contiguity of the Gelfand HGF, see \cite{Horikawa,Kimura-H-T,sasaki}.

We expect that the Gelfand HGF on $\mathrm{GM}(r,N)$ for various partitions
$\lambda$ of $N$, except for $\lambda=(N)$, gives a solution of the 2dTHE.
This problem is discussed in another paper.

This paper is organized as follows. In Section 2, we recall the facts
on the relation between the Laplace sequence of hyperbolic operators
and the 2-dimensional Toda equation satisfied by the invariants. The
link to 2dTHE is also discussed. For the EPD operator, we compute
the invariants of the operators of Laplace sequence and determine
the explicit form of the particular solution to the 2dTHE. In Section
3, we give the definition of the Gelfand HGF of non-confluent type
on $\mathrm{GM}(2,N)$ and discuss its covariance with respect to the group
action $\mathrm{GL}(2)\curvearrowright Z\curvearrowleft H$. We also give the
system of differential equations characterizing $F(z;\alpha)$, which
will be called the Gelfand hypergeometric system (HGS). We show that
the system of EPD equations is obtained as a result of reduction of
the Gelfand HGS (Proposition \ref{prop:red-EPD}) and the contiguity
operators for the EPD system are obtained from those for the Gelfand
HGF. In Section 4, after studying the contiguity structure of the
system of EPDs, we combine them with the result in Section 2 to give
Theorem \ref{thm:main}, the main theorem of this paper.
\begin{acknowledgement*}
I thank Professor Kazuo Okamoto for valuable comments. I would like
to dedicate this article to late Professor Masatoshi Noumi, who is
my friend and gives me advices constantly. This work was supported
by JSPS KAKENHI Grant Number JP19K03521.
\end{acknowledgement*}

\section{Laplace sequence and Toda lattice}

\subsection{\label{subsec:Laplace-seq}Generality of Laplace sequence}

Let $x,y$ be the complex coordinates of $\mathbb{C}^{2}$, $\Omega\subset\mathbb{C}^{2}$
be a simply connected domain, and $\mathcal{O}(\Omega)$ be the set of holomorphic
functions on $\Omega$. We consider the following hyperbolic differential
equation:
\begin{equation}
Mu=\left(\partial_{x}\partial_{y}+a(x,y)\partial_{x}+b(x,y)\partial_{y}+c(x,y)\right)u=0,\label{eq:laplace-1}
\end{equation}
where $\partial_{x}=\partial/\partial x,\partial_{y}=\partial/\partial y$ and $a,b,c\in \mathcal{O}(\Omega)$.
Write the operator $M$ in the form
\begin{equation}
M=(\partial_{x}+b)(\partial_{y}+a)-h,\label{eq:laplace-1-1}
\end{equation}
 or 
\begin{equation}
M=(\partial_{y}+a)(\partial_{x}+b)-k,\label{eq:laplace-1-2}
\end{equation}
where
\[
h=a_{x}+ab-c,\quad k=b_{y}+ab-c
\]
with $a_{x}=\partial a/\partial x,b_{y}=\partial b/\partial y$. Then functions $h=h(x,y),k=k(x,y)$
are called the \emph{invariants} of the operator $M$. The meaning
of ``invariant'' comes from the following fact, which is easily
shown by direct computation.
\begin{lem}
\label{lem:normal-form} For the operator $M$ above and  a function
$f\in\mathcal{O}(\Omega)$ such that $1/f\in\mathcal{O}(\Omega)$, define the operator $M'$
by
\[
M'=f^{-1}\cdot M\cdot f=\partial_{x}\partial_{y}+a'\partial_{x}+b'\partial_{y}+c'.
\]
Then $M'$ is given in terms of $F=\log f$ by
\begin{align}
a' & =a+F_{y},\nonumber \\
b' & =b+F_{x},\label{eq:laplace-1-4}\\
c' & =c+aF_{x}+bF_{y}+F_{x,y}+F_{x}F_{y}.\nonumber 
\end{align}
The invariants of $M$ coincide with those of $M'$. 
\end{lem}

\begin{lem}
\label{lem:Lap_+}For the operator $M$ given by (\ref{eq:laplace-1}),
assume that $h(x,y)\neq0$ for any $(x,y)\in\Omega$. Then, by the change
of unknown $u\mapsto u_{+}$: 
\begin{equation}
u_{+}=\mathscr{L}_{+}u:=(\partial_{y}+a)u,\label{eq:laplace-1-3}
\end{equation}
the equation (\ref{eq:laplace-1}) is transformed to 
\begin{equation}
M_{+}u_{+}=\left(\partial_{x}\partial_{y}+a_{+}\partial_{x}+b_{+}\partial_{y}+c_{+}\right)u_{+}=0,\label{eq:laplace-2}
\end{equation}
 where
\begin{align}
a_{+} & =a-\partial_{y}\log h,\nonumber \\
b_{+} & =b,\label{eq:laplace-3}\\
c_{+} & =c-a_{x}+b_{y}-b\,\partial_{y}\log h.\nonumber 
\end{align}
The invariants $h_{+},k_{+}$ of $M_{+}$ are related to those of
$M$ as 

\begin{equation}
h_{+}=2h-k-\partial_{x}\partial_{y}\log h,\quad k_{+}=h.\label{eq:laplace-3-0}
\end{equation}
\end{lem}

\begin{proof}
We give a brief sketch of the proof. See also \cite{Darboux}. By
virtue of the expression (\ref{eq:laplace-1-1}) of $M$, the equation
(\ref{eq:laplace-1}) can be written as 

\begin{equation}
\partial_{x}u_{+}+bu_{+}-hu=0.\label{eq:laplace-3-1}
\end{equation}
 Differentiate it with respect to $y$, and eliminate $u$ and $\partial_{y}u$
from the resulted equation by the help of (\ref{eq:laplace-1-3})
and (\ref{eq:laplace-3-1}). Then we get the equation (\ref{eq:laplace-2})
with (\ref{eq:laplace-3}). The invariants $h_{+},k_{+}$ are computed
using (\ref{eq:laplace-3}), and (\ref{eq:laplace-3-0}) is obtained.
\end{proof}
Using the expression (\ref{eq:laplace-1-2}) for $M$, we can obtain
the following result in a similar way as in Lemma \ref{lem:Lap_+}.

\begin{lem}
\label{lem:Lap_-}For the operator $M$ given by (\ref{eq:laplace-1}),
assume that $k(x,y)\neq0$ for any $(x,y)\in\Omega$. Then, by the change
of unknown $u\mapsto u_{-}$:
\begin{equation}
u_{-}=\mathscr{L}_{-}u:=(\partial_{x}+b)u,\label{eq:laplace-3-1-1}
\end{equation}
 the equation (\ref{eq:laplace-1}) is transformed to 
\[
M_{-}u_{-}=\left(\partial_{x}\partial_{y}+a_{-}\partial_{x}+b_{-}\partial_{y}+c_{-}\right)u_{-}=0,
\]
where 
\begin{align}
a_{-} & =a,\nonumber \\
b_{-} & =b-\partial_{x}\log k,\label{eq:laplace-3-2}\\
c_{-} & =c+a_{x}-b_{y}-a\,\partial_{x}\log k.\nonumber 
\end{align}
 The invariants $h_{-},k_{-}$ of $M_{-}$ are related to those of
$M$ as

\begin{equation}
h_{-}=k,\quad k_{-}=2k-h-\partial_{x}\partial_{y}\log k.\label{eq:laplace3-3}
\end{equation}
\end{lem}

The expressions (\ref{eq:laplace-1-1}) and (\ref{eq:laplace-1-2})
for $M$ imply that 
\begin{equation}
\left(\mathscr{L}_{-}\circ\mathscr{L}_{+}\right)u=h\cdot u,\quad\left(\mathscr{L}_{+}\circ\mathscr{L}_{-}\right)u=k\cdot u\label{eq:laplace-3-4}
\end{equation}
holds for any solution $u$ of $Mu=0$.

As a consequence of Lemmas \ref{lem:Lap_+}, \ref{lem:Lap_-}, we
have the following sequence of hyperbolic differential operators starting
from $M_{0}:=M$: 
\begin{equation}
\cdots\leftarrow M_{-n}\leftarrow\cdots\leftarrow M_{-1}\leftarrow M_{0}\to M_{1}\to\cdots\to M_{n}\to\cdots,\label{eq:laplace-3-3}
\end{equation}
where, for $n\geq1$, $M_{n}$ is obtained from $M_{n-1}$ by applying
Lemma \ref{lem:Lap_+} under the condition that the invariant $h$
of $M_{n-1}$ satisfies $h\neq0$, and $M_{-n}$ is obtained from
$M_{-n+1}$ by applying Lemma \ref{lem:Lap_-} under the condition
that the invariant $k$ of $M_{-n+1}$ satisfies $k\neq0$. The sequence
$\{M_{n}\}_{n\geq0}$ or $\{M_{n}\}_{n\leq0}$ is called the \emph{Laplace
sequence} obtained from $M_{0}$. The invariants of $M_{n}$ will
be denoted as $h_{n},k_{n}$. In considering the Laplace sequence,
we tacitly assume that the invariants do not vanish.

The following results are the consequences of Lemmas \ref{lem:Lap_+},
\ref{lem:Lap_-}.
\begin{prop}
\label{prop:lap-1}For the Laplace sequence $\{M_{n}\}_{n\in\mathbb{Z}_{\geq0}}$,
$M_{n}=\partial_{x}\partial_{y}+a_{n}\partial_{x}+b_{n}\partial_{y}+c_{n}$, the operator
$M_{n+1}$ and its invariants are determined from $M_{n}$ as 

\begin{align}
a_{n+1} & =a_{n}-\partial_{y}\log h_{n},\nonumber \\
b_{n+1} & =b_{n},\label{eq:laplace-4}\\
c_{n+1} & =c_{n}-\partial_{x}a_{n}+\partial_{y}b_{n}-b_{n}\partial_{y}\log h_{n},\nonumber 
\end{align}
and 

\begin{equation}
h_{n+1}=2h_{n}-k_{n}-\partial_{x}\partial_{y}\log h_{n},\;k_{n+1}=h_{n}.\label{eq:laplace-4-1}
\end{equation}
\end{prop}

\begin{prop}
\label{prop:lap-2}For the Laplace sequence $\{M_{n}\}_{n\in\mathbb{Z}_{\leq0}}$,
$M_{n}=\partial_{x}\partial_{y}+a_{n}\partial_{x}+b_{n}\partial_{y}+c_{n}$, the operator
$M_{n-1}$ and its invariants are determined from $M_{n}$ as 

\begin{align}
a_{n-1} & =a_{n},\nonumber \\
b_{n-1} & =b_{n}-\partial_{x}\log k_{n},\label{eq:laplace-5}\\
c_{n-1} & =c_{n}+\partial_{x}a_{n}-\partial_{y}b_{n}-a_{n}\partial_{x}\log k_{n},\nonumber 
\end{align}
 and 

\begin{equation}
h_{n-1}=k_{n},\;k_{n-1}=2k_{n}-h_{n}-\partial_{x}\partial_{y}\log k_{n}.\label{eq:laplace-5-1}
\end{equation}
 
\end{prop}

Put $r_{n}=-k_{n}=-h_{n-1}$. Then the relations (\ref{eq:laplace-4-1})
and (\ref{eq:laplace-5-1}) are expressed as 
\begin{equation}
\partial_{x}\partial_{y}\log r_{n}=r_{n+1}-2r_{n}+r_{n-1},\quad n\in\mathbb{Z}.\label{eq:toda-2}
\end{equation}
 This recurrence relation is called the \emph{2-dimensional Toda equation}
(2dTE). In Section \ref{subsec:2DTE-laplace}, we consider another
form of 2dTE.

\subsection{\label{subsec:The-sequence-normal}Sequence of hyperbolic operators
of the normal form}

To relate the Laplace sequence to another form of the 2dTE, we discuss
the reduction of the operator 
\begin{equation}
M=\partial_{x}\partial_{y}+a\partial_{x}+b\partial_{y}+c,\label{eq:laplace-6-1}
\end{equation}
to the normal form $M'=\partial_{x}\partial_{y}+a'\partial_{x}+c'$, which is obtained
from (\ref{eq:laplace-6-1}) by eliminating the term $b\partial_{y}$ by
considering $M'=f^{-1}\cdot M\cdot f$ with an appropriate function
$f$. This corresponds to consider the change of unknown $u\to v=f^{-1}u$
for the system $Mu=0$. To find such $f$, note the expression (\ref{eq:laplace-1-4})
for $b'$ in Lemma \ref{lem:normal-form}.
\begin{lem}
\label{lem:norm-1}Take $f$ satisfying $b+\partial_{x}\log f=0$, namely,
$f=\exp F,F=-\int^{x}b(t,y)dt$. Then $M'=f^{-1}\cdot M\cdot f$ has
the form $M'=\partial_{x}\partial_{y}+a'\partial_{x}+c'$ with 
\[
a'=a+F_{y},\quad c'=c+aF_{x}+bF_{y}+F_{x,y}+F_{x}F_{y}.
\]
In this case $a'$ and $c'$ are related to the invariants $h,k$
as 
\begin{equation}
a'_{x}=h-k,\quad c'=-k.\label{eq:laplace-8}
\end{equation}
\end{lem}

\begin{proof}
The expressions for $a',c'$ follows from Lemma \ref{lem:normal-form}.
Since the invariants of $M'$ are the same as those for $M$ by Lemma
\ref{lem:normal-form}, we have 
\[
h=a'_{x}-a'b'-c'=a'_{x}-c',\quad k=b_{y}'-a'b'-c'=-c',
\]
which give (\ref{eq:laplace-8}).
\end{proof}
Suppose we are given $M_{0}$ in the normal form $M_{0}=\partial_{x}\partial_{y}+a_{0}\partial_{x}+c_{0}$.
We construct the sequence $\{M_{n}\}_{n\in\mathbb{Z}}$ consisting of the
operators 
\[
M_{n}=\partial_{x}\partial_{y}+a_{n}\partial_{x}+c_{n},
\]
such that $M_{n+1}$ is obtained from $M_{n}$ by the process given
in Lemma \ref{lem:Lap_+}. Let $\{M_{n}\}_{n\geq0}$ be the Laplace
sequence constructed from $M_{0}$ in the normal form by virtue of
Proposition \ref{prop:lap-1}. Then (\ref{eq:laplace-4}) implies
that the operators $M_{n}$ for $n\geq0$ are in the normal form and
satisfy our requirement. But the operators, constructed from $M_{0}$
applying Proposition \ref{prop:lap-2}, is not necessarily in the
normal form. So we construct the operators for $n<0$ step by step.
We construct $M_{-1}=\partial_{x}\partial_{y}+a_{-1}\partial_{x}+c_{-1}$ from $M_{0}$
as follows. Apply Lemma \ref{lem:Lap_-} to $M_{0}$ to obtain 
\[
M_{-}=\partial_{x}\partial_{y}+a_{-}\partial_{x}+b_{-}\partial_{y}+c_{-},
\]
where
\begin{align*}
a_{-} & =a_{0},\\
b_{-} & =-\partial_{x}\log k_{0},\\
c_{-} & =c_{0}+\partial_{x}a_{0}-a_{0}\,\partial_{x}\log k_{0}.
\end{align*}
 Then applying Lemma \ref{lem:norm-1}, we take $M_{-}$ to the normal
form
\[
M_{-1}=\partial_{x}\partial_{y}+a_{-1}\partial_{x}+c_{-1},
\]
where, using $F=\log k_{0}$, the coefficients are given by
\begin{align*}
a_{-1} & =a_{-}+\partial_{y}\log k_{0}=a_{0}+F_{y},\\
c_{-1} & =c_{-}+(a_{-})F_{x}+(b_{-})F_{y}+F_{x,y}+F_{x}F_{y}=c_{0}+\partial_{x}a_{0}+F_{x,y}.
\end{align*}
We should check that $M_{0}$ is obtained from $M_{-1}$ by the process
in Lemma \ref{lem:Lap_+}. This is easily checked as follows. Let
us denote the operator obtained from $M_{-1}$ by the process in Lemma
\ref{lem:Lap_+} as $M_{0}'=\partial_{x}\partial_{y}+a'\partial_{x}+c'$. Then $a'$
and $c'$ is obtained from $M_{-1}$ as
\begin{align*}
a' & =a_{-1}-\partial_{y}\log h_{-1}=a_{0}+F_{y}-\partial_{y}\log h_{-1}=a_{0},\\
c' & =c_{-1}-\partial_{x}a_{-1}=c_{0}+\partial_{x}a_{0}+F_{x,y}-\partial_{x}(a_{0}+F_{y})=c_{0}.
\end{align*}
 Here we used $h_{-1}=k_{0}.$ Repeating this construction successively,
we obtain the sequence $M_{0}\to M_{-1}\to M_{-2}\to\cdots$ of the
normal form which satisfy our requirement that $M_{n+1}$ is obtained
from $M_{n}$ by the process in Lemma \ref{lem:Lap_+} for any $n\leq-1$.
Thus we have proved the following.
\begin{prop}
\label{prop:lap-3}From the given $M_{0}=\partial_{x}\partial_{y}+a_{0}\partial_{x}+c_{0}$,
we can construct the sequence of hyperbolic operators of the normal
form
\[
M_{n}=\partial_{x}\partial_{y}+a_{n}\partial_{x}+c_{n},\quad n\in\mathbb{Z}
\]
 such that $M_{n+1}$ is obtained from $M_{n}$ by 
\begin{align*}
a_{n+1} & =a_{n}-\partial_{y}\log h_{n},\\
c_{n+1} & =c_{n}-\partial_{x}a_{n}
\end{align*}
under the condition that the invariant $h_{n}$ of $M_{n}$ is not
zero for any $n$.
\end{prop}

The sequence $\{M_{n}\}_{n\in\mathbb{Z}}$ obtained in Proposition \ref{prop:lap-3}
is also called the Laplace sequence.

\subsection{\label{subsec:2DTE-laplace}2dTE arising from the Laplace sequence}

In this section, changing the notation used in Section \ref{subsec:The-sequence-normal},
we write the Laplace sequence $\{M_{n}\}_{n\in\mathbb{Z}}$ of the normal
form as
\begin{equation}
M_{n}=\partial_{x}\partial_{y}+s_{n+1}\partial_{x}+r_{n}.\label{eq:toda-4}
\end{equation}
Then (\ref{eq:laplace-8}) of Lemma \ref{lem:norm-1} says that the
invariants $h_{n},k_{n}$ are given by
\[
h_{n}=\partial_{x}s_{n+1}-r_{n},\quad k_{n}=-r_{n}.
\]
Since $\{M_{n}\}_{n\in\mathbb{Z}}$ is the Laplace sequence, (\ref{eq:laplace-4})
of Proposition \ref{prop:lap-1} implies $s_{n+1}=s_{n}-\partial_{y}\log h_{n-1}$.
Thus, noting $h_{n-1}=k_{n}$, we have the recurrence relation for
the pair $(s_{n+1},r_{n})$:
\begin{equation}
\begin{cases}
\partial_{x}s_{n+1}=r_{n}-r_{n+1},\\
\partial_{y}\log r_{n}=s_{n}-s_{n+1}.
\end{cases}\label{eq:toda-3}
\end{equation}
Conversely, the following result is known and is easily shown.
\begin{prop}
\label{prop:Toda-1}If $\{(s_{n+1},r_{n})\}_{n\in\mathbb{Z}}$ satisfies (\ref{eq:toda-3}),
then the sequence $\{M_{n}\}_{n\in\mathbb{Z}}$ defined by (\ref{eq:toda-4})
is the Laplace sequence.
\end{prop}

We mainly consider the 2-dimensional Toda-Hirota equation (2dTHE):

\begin{equation}
\partial_{x}\partial_{y}\log\tau_{n}=\frac{\tau_{n+1}\tau_{n-1}}{\tau_{n}^{2}},\quad n\in\mathbb{Z}.\label{eq:toda-1}
\end{equation}

The following result gives the link between 2dTHE and 2dTE, which
is well known and is easily verified.
\begin{prop}
\label{prop:Toda-2}Let $\{\tau_{n}(x,y)\}_{n\in\mathbb{Z}}$ satisfy the
equation (\ref{eq:toda-1}) and let $(s_{n+1},r_{n})$ be defined
by
\begin{equation}
s_{n+1}=\partial_{y}\log\left(\frac{\tau_{n}}{\tau_{n+1}}\right),\quad r_{n}=\partial_{x}\partial_{y}\log\tau_{n},\label{eq:toda-1-1}
\end{equation}
 then $\{(s_{n+1},r_{n})\}_{n\in\mathbb{Z}}$ gives a solution of (\ref{eq:toda-3}),
and $\{r_{n}\}_{n\in\mathbb{Z}}$ satisfies the 2dTE (\ref{eq:toda-2}).
\end{prop}

\begin{proof}
Differentiate $s_{n+1}=\partial_{y}\log\left(\frac{\tau_{n}}{\tau_{n+1}}\right)$
with respect to $x$, we have 
\[
\partial_{x}s_{n+1}=\partial_{x}\partial_{y}\log\left(\frac{\tau_{n}}{\tau_{n+1}}\right)=\partial_{x}\partial_{y}\log\tau_{n}-\partial_{x}\partial_{y}\log\tau_{n+1}=r_{n}-r_{n+1}.
\]
Similarly
\[
\partial_{y}\log r_{n}=\partial_{y}\log\left(\partial_{x}\partial_{y}\log\tau_{n}\right)=\partial_{y}\log\left(\frac{\tau_{n+1}\tau_{n-1}}{\tau_{n}^{2}}\right)=s_{n}-s_{n+1}.
\]
The last assertion follows as $\partial_{x}\partial_{y}\log r_{n}=\partial_{x}s_{n}-\partial_{x}s_{n+1}=r_{n-1}-2r_{n}+r_{n+1}$.
\end{proof}

\subsection{\label{subsec:Darboux-transformation}B\"acklund transformation }

When a solution $\{t_{n}\}_{n\in\mathbb{Z}}$ of 2dTHE is given, we can construct
a new solution of the 2dTHE as explained in the following. Proposition
\ref{prop:Toda-1} tells us that $\{(s_{n+1},r_{n})\}_{n\in\mathbb{Z}}$,
defined by (\ref{eq:toda-1-1}) taking $t_{n}$ as $\tau_{n}$, gives
a Laplace sequence $\{M_{n}\}_{n\in\mathbb{Z}}$ of the form

\[
M_{n}=\partial_{x}\partial_{y}+s_{n+1}\partial_{x}+r_{n}.
\]
Let $\Omega$ be a simply connected domain where $M_{n}$ are holomorphically
defined, and let $\mathcal{S}(n)$ be the space of holomorphic solutions of
$M_{n}u=0$ in $\Omega$. We have the differential operators $H_{n}$
and $B_{n}$ which give linear maps
\[
H_{n}:\mathcal{S}(n)\to\mathcal{S}(n+1),\quad B_{n}:\mathcal{S}(n)\to\mathcal{S}(n-1)
\]
satisfying $B_{n+1}H_{n}=1$ and $H_{n-1}B_{n}=1$ on $\mathcal{S}(n)$. They
are given by 
\begin{equation}
H_{n}=\partial_{y}+s_{n+1},\quad B_{n}=-r_{n}^{-1}\partial_{x}\label{eq:darboux-1}
\end{equation}
as is seen from (\ref{eq:laplace-3-4}) and the construction of the
operators in Proposition \ref{prop:lap-3}. 
\begin{prop}
\label{prop:Backlund-1}Assume that the invariants $r_{n}(=-h_{n-1})$
are nonzero for any $n\in\mathbb{Z}$. Then, for any $n$, $H_{n}$ and $B_{n}$
above define the linear isomorphisms
\[
H_{n}:\mathcal{S}(n)\to\mathcal{S}(n+1),\quad B_{n}:\mathcal{S}(n)\to\mathcal{S}(n-1).
\]
 
\end{prop}

If we are given a solution $u_{0}$ of $M_{0}u=0$, we can construct
$\{u_{n}\}_{n\in\mathbb{Z}}$ such that $u_{n}\in\mathcal{S}(n)$ satisfying $u_{n+1}=H_{n}u_{n}$
and $u_{n-1}=B_{n}u_{n}$. The following is the important result to
establish our main result which assert that the Gelfand HGF gives
a solution to the 2dTHE. 

\begin{prop}
\label{prop:Backlund-2}Suppose that $\{t_{n}\}_{n\in\mathbb{Z}}$ is a solution
of 2dTHE (\ref{eq:toda-1}) from which the Laplace sequence $\{M_{n}\}_{n\in\mathbb{Z}}$
is constructed. Given $\{u_{n}\}_{n\in\mathbb{Z}}$ such that $u_{n}\in\mathcal{S}(n)$
and satisfies $u_{n+1}=H_{n}u_{n}$ and $u_{n-1}=B_{n}u_{n}$. Define
$\{\tau_{n}\}_{n\in\mathbb{Z}}$ by $\tau_{n}=t_{n}u_{n}.$ Then $\{\tau_{n}\}_{n\in\mathbb{Z}}$
gives a solution of the 2dTHE (\ref{eq:toda-1}).
\end{prop}

\begin{proof}
By definition, we have
\[
\partial_{x}\partial_{y}\log\tau_{n}=\partial_{x}\partial_{y}\log t_{n}+\partial_{x}\partial_{y}\log u_{n}=r_{n}+\partial_{x}\partial_{y}\log u_{n}.
\]
For this $u_{n}$ we show 
\begin{equation}
\partial_{x}\partial_{y}\log u_{n}=\frac{r_{n}u_{n+1}u_{n-1}}{u_{n}^{2}}-r_{n}.\label{eq:Back-1}
\end{equation}
Noting $\partial_{x}\partial_{y}\log u_{n}=\partial_{x}\partial_{y}u_{n}/u_{n}-\partial_{x}u_{n}\cdot\partial_{y}u_{n}/u_{n}^{2}$
and using $H_{n}=\partial_{y}+s_{n+1},B_{n}=-r_{n}^{-1}\partial_{x}$, we compute
\begin{align*}
\partial_{x}\partial_{y}u_{n} & =-s_{n+1}\partial_{x}u_{n}-r_{n}u_{n}=-s_{n+1}\cdot(-r_{n}B_{n}u_{n})-r_{n}u_{n}\\
 & =r_{n}s_{n+1}u_{n-1}-r_{n}u_{n},\\
\partial_{x}u_{n}\cdot\partial_{y}u_{n} & =(-r_{n}B_{n}u_{n})\cdot(H_{n}u_{n}-s_{n+1}u_{n})\\
 & =-r_{n}u_{n+1}u_{n-1}+r_{n}s_{n+1}u_{n-1}u_{n}.
\end{align*}
Hence we have (\ref{eq:Back-1}). It follows that 
\begin{align*}
\partial_{x}\partial_{y}\log\tau_{n} & =r_{n}\frac{u_{n+1}u_{n-1}}{u_{n}^{2}}=\partial_{x}\partial_{y}\log t_{n}\cdot\frac{u_{n+1}u_{n-1}}{u_{n}^{2}}\\
 & =\frac{t_{n+1}t_{n-1}}{t_{n}^{2}}\cdot\frac{u_{n+1}u_{n-1}}{u_{n}^{2}}=\frac{\tau_{n+1}\tau_{n-1}}{\tau_{n}^{2}}.
\end{align*}
\end{proof}

\subsection{\label{subsec:EPD}Euler-Poisson-Darboux equation and a solution
of the 2dTHE}

To recognize the Gelfand HGF as a particular solution to 2dTHE, it
is important to find a seed solution of 2dTHE. We use the seed solution
arising from the Laplace sequence of the so-called Euler-Poisson-Darboux
equation (EPD equation), which is 

\begin{equation}
M_{0}u:=\left(\partial_{x}\partial_{y}+\frac{\beta}{x-y}\partial_{x}+\frac{\alpha}{y-x}\partial_{y}\right)u=0,\label{eq:epd-0}
\end{equation}
where $\alpha,\beta$ are complex constants. The normal form of $M_{0}$
is 
\begin{equation}
M_{0}'=\partial_{x}\partial_{y}+\frac{\beta-\alpha}{x-y}\partial_{x}+\frac{\alpha(\beta+1)}{(x-y)^{2}}.\label{eq:epd-2}
\end{equation}
 By the process described in Section \ref{subsec:The-sequence-normal},
we can construct the Laplace sequence $\{M_{n}'\}_{n\in\mathbb{Z}}$ starting
from $M_{0}'$ and a solution of 2dTE associated with it. It is easily
seen that 
\begin{align*}
M_{0} & =\left(\partial_{x}+\frac{\alpha}{y-x}\right)\left(\partial_{y}+\frac{\beta}{x-y}\right)-h_{0},\\
 & =\left(\partial_{y}+\frac{\beta}{x-y}\right)\left(\partial_{x}+\frac{\alpha}{y-x}\right)-k_{0},
\end{align*}
where 
\begin{equation}
h_{0}=-\frac{(\alpha+1)\beta}{(x-y)^{2}},\quad k_{0}=-\frac{\alpha(\beta+1)}{(x-y)^{2}}.\label{eq:epd-2-1}
\end{equation}

\begin{lem}
For the Laplace sequence $\{M_{n}'\}_{n\in\mathbb{Z}}$, the invariants $h_{n},k_{n}$
are given by
\begin{equation}
h_{n}=-\frac{(\alpha+n+1)(\beta-n)}{(x-y)^{2}},\quad k_{n}=h_{n-1}.\label{eq:EPD-0}
\end{equation}
\end{lem}

\begin{proof}
Since the invariants of $M_{0}$ and $M_{0}'$ are the same and given
by (\ref{eq:epd-2-1}), and since the invariants $h_{n},k_{n}$ of
$M_{n}'$ satisfy the relation 
\begin{equation}
\partial_{x}\partial_{y}\log h_{n}=-h_{n+1}+2h_{n}-h_{n-1},\quad n\in\mathbb{Z}\label{eq:EPD-0-1}
\end{equation}
and $k_{n}=h_{n-1}$, we determine $\{h_{n}\}_{n\in\mathbb{Z}}$ by the recurrence
relation (\ref{eq:EPD-0-1}) with the initial condition (\ref{eq:epd-2-1}).
For $n\geq0$, we use (\ref{eq:EPD-0-1}) in the form
\begin{equation}
(h_{n+1}-h_{n})-(h_{n}-h_{n-1})=-\partial_{x}\partial_{y}\log h_{n}.\label{eq:EPD-1}
\end{equation}
Put $h_{n}=-A_{n}/(x-y)^{2}$, then $A_{0}=(\alpha+1)\beta,$ $A_{-1}=\alpha(\beta+1)$.
Moreover, put $B_{n}:=A_{n}-A_{n-1}\;(n\geq0)$. Since 
\[
\partial_{x}\partial_{y}\log h_{n}=\partial_{x}\partial_{y}\log\left(-\frac{A_{n}}{(x-y)^{2}}\right)=-\frac{2}{(x-y)^{2}},
\]
 (\ref{eq:EPD-1}) reads $B_{n+1}-B_{n}=-2\;(n\geq0)$, and we have
$B_{n}=-2n+B_{0}=-2n+(\beta-\alpha)$, i.e., $A_{n}-A_{n-1}=-2n+(\beta-\alpha)$.
Solving this equation with the initial condition $A_{0}=(\alpha+1)\beta$,
we have $A_{n}=-(n+\alpha+1)(n-\beta)$. For the case $n\leq-1$, we use
(\ref{eq:EPD-1}) in the form
\[
(h_{n-1}-h_{n})-(h_{n}-h_{n+1})=-\partial_{x}\partial_{y}\log h_{n}.
\]
 Solving this recurrence relation for $n$ in the decreasing direction,
we see that $h_{n}$ for $n\leq-1$ are given also by (\ref{eq:EPD-0}).
\end{proof}
\begin{prop}
\label{prop:Laplace-2}Let $M_{0}$ and $M_{0}'$ be given by (\ref{eq:epd-0})
and (\ref{eq:epd-2}) and assume $\alpha,\beta\notin\mathbb{Z}$.

(1) The Laplace sequence $\{M_{n}'\}_{n\in\mathbb{Z}}$ of normal form arising
from $M_{0}'$ is given by
\begin{equation}
M_{n}'=\partial_{x}\partial_{y}+\frac{\beta-\alpha-2n}{x-y}\partial_{x}+\frac{(\alpha+n)(\beta-n+1)}{(x-y)^{2}}\label{eq:EPD-2}
\end{equation}
with the invariants  $h_{n}=-(\alpha+n+1)(\beta-n)/(x-y)^{2},k_{n}=h_{n-1}$.
The EPD operator $M_{n}$ with the normal form $M_{n}'$ is given
by 
\[
M_{n}=\partial_{x}\partial_{y}+\frac{\beta-n}{x-y}\partial_{x}+\frac{\alpha+n}{y-x}\partial_{y}.
\]

(2\textup{) }The solution\textup{ }of the 2dTE\textup{ (\ref{eq:toda-3})
associated with the Laplace sequence $\{M_{n}'\}_{n\in\mathbb{Z}}$ is 
\[
(s_{n+1},r_{n})=\left(\frac{\beta-\alpha-2n}{x-y},\frac{(\alpha+n)(\beta-n+1)}{(x-y)^{2}}\right).
\]
}

(3) \textup{$r_{n}=(\alpha+n)(\beta-n+1)/(x-y)^{2}$ }gives a solution\textup{
}to\textup{ $\partial_{x}\partial_{y}\log r_{n}=r_{n+1}-2r_{n}+r_{n-1}$.}
\end{prop}

As for the 2dTHE (\ref{eq:toda-1}), we have the following.
\begin{prop}
\label{prop:Laplace-3}Assume $\alpha,\beta\notin\mathbb{Z}$. Then the Laplace sequence
$\{M_{n}'\}_{n\in\mathbb{Z}}$ given by (\ref{eq:EPD-2}) provides a solution
\[
t_{n}(x,y)=B(\alpha,\beta;n)(x-y)^{p(\alpha,\beta;n)}
\]
of the 2dTHE (\ref{eq:toda-1}), where 
\[
p(\alpha,\beta;n)=(\alpha+n)(\beta-n+1),
\]
and $B(\alpha,\beta;n)$ is given, under the condition $B(0)=1,B(1)=A$ with
an arbitrary constant $A$, by
\[
B(\alpha,\beta;n)=\begin{cases}
A^{n}\prod_{k=0}^{n-1}\left(\prod_{l=1}^{k}p(\alpha,\beta;l)\right), & n\geq2,\\
A^{n}\prod_{k=1}^{|n|}\left(\prod_{l=-k+1}^{0}p(\alpha,\beta;l)\right), & n\leq-1.
\end{cases}
\]
\end{prop}

\begin{proof}
Let us determine $t_{n}$ by the equation
\begin{equation}
\partial_{x}\partial_{y}\log t_{n}=\frac{t_{n+1}t_{n-1}}{t_{n}^{2}}.\label{eq:EPD-3}
\end{equation}
Recall that $r_{n}$ and $t_{n}$ are related as 
\begin{equation}
\partial_{x}\partial_{y}\log t_{n}=r_{n}=\frac{p(n)}{(x-y)^{2}},\quad p(n)=(\alpha+n)(\beta-n+1).\label{eq:EPD-6}
\end{equation}
Noting $\partial_{x}\partial_{y}\log(x-y)=1/(x-y)^{2}$, the condition (\ref{eq:EPD-6})
is written as $\partial_{x}\partial_{y}\log t_{n}=p(n)\partial_{x}\partial_{y}\log(x-y)$,
namely, $\partial_{x}\partial_{y}\left(\log t_{n}-p(n)\log(x-y)\right)=0$.
So we find $t_{n}$ in the form $t_{n}=B(n)(x-y)^{p(n)}$, where $B(n)$
is the constant independent of $x,y$. Put this expression in the
equation (\ref{eq:EPD-3}) and note that $p(n+1)-2p(n)+p(n-1)=-2$,
then we have 
\begin{equation}
\frac{B(n+1)B(n-1)}{B(n)^{2}}=p(n).\label{eq:EPD-5}
\end{equation}
To determine $B(n)$ for $n\geq2$ under the condition $B(0)=1,B(1)=A$,
we use (\ref{eq:EPD-5}) in the form
\begin{equation}
\frac{B(n+1)}{B(n)}/\frac{B(n)}{B(n-1)}=p(n).\label{eq:EPD-4}
\end{equation}
It follows that 
\[
\frac{B(n+1)}{B(n)}=\frac{B(1)}{B(0)}\prod_{k=1}^{n}p(k)=A\prod_{k=1}^{n}p(k),
\]
Thus we obtain $B(n)$ as 
\[
B(n)=\prod_{k=0}^{n-1}\left(A\prod_{l=1}^{k}p(l)\right)=(-1)^{\frac{n(n-1)}{2}}A^{n}\prod_{k=1}^{n-1}(\alpha+1)_{k}(-\beta)_{k}.
\]
 To determine $B(n)$ for $n\leq-1$, we use (\ref{eq:EPD-5}) in
the form 
\[
\frac{B(n-1)}{B(n)}/\frac{B(n)}{B(n+1)}=p(n).
\]
Then in a similar way as in the case $n\geq2$, we have the expression
of $B(n)$ for $n\leq-1$ as given in the proposition.
\end{proof}

\section{EPD arising from the Gelfand HGF}

In this section, we recall the facts on the Gelfand HGF on the complex
Grassmannian manifold $\mathrm{GM}(2,N)$ and show that the system of EPD
equations are obtained naturally from the system of differential equations
characterizing the Gelfand HGF as a consequence of reduction by the
action of Cartan subgroup of $\mathrm{GL}(N) $.

\subsection{Definition of Gelfand's HGF}

Let $N\geq3$ be an integer, $G=\mathrm{GL}(N)$ be the complex general linear
group consisting of nonsingular matrices of size $N$, and let $H$
be the Cartan subgroup of $G$:
\[
H=\left\{ h=\mathrm{diag}(h_{1},\dots,h_{N})\mid h_{i}\in\mathbb{C}^{\times}\right\} \subset G.
\]
The Lie algebra of $G$ and $H$ will be denoted as $\mathfrak{g}$ and $\mathfrak{h}$,
respectively. $\mathfrak{h}$ is a Cartan subalgebra of $\mathfrak{g}$. We restrict
ourselves to consider the Gelfand HGF defined by $1$-dimensional
integral to discuss its relation to the 2dTHE. The Gelfand HGF is
defined as a Radon transform of a character of the universal covering
group $\tilde{H}$ of $H$. Since $H\simeq(\mathbb{C}^{\times})^{N}$ and a character
of $\tilde{\mathbb{C}}^{\times}$ is given by a complex power function $x\mapsto x^{a}$
for some $a\in\mathbb{C}$, the characters of $\tilde{H}$ are given as follows.
\begin{lem}
A character $\chi:\tilde{H}\to\mathbb{C}^{\times}$ is given by 
\[
\chi(h;\alpha)=\prod_{j=1}^{N}h_{j}^{\alpha_{j}},\quad h=\mathrm{diag}(h_{1},\dots,h_{N})
\]
 for some $\alpha=(\alpha_{1},\dots,\alpha_{N})\in\mathbb{C}^{N}.$
\end{lem}

Note that $\alpha$ is regarded as an element of $\mathfrak{h}^{*}=\mathrm{Hom}_{\mathbb{C}}(\mathfrak{h},\mathbb{C})$
such that $\alpha(E_{p,p})=\alpha_{p}$ for the $(p,p)$-th matrix unit $E_{p,p}\in\mathfrak{h}$.
For the character $\chi(\cdot;\alpha)$ we assume the conditions $\alpha_{j}\notin\mathbb{Z}\;(1\leq j\leq N)$
and
\begin{equation}
\alpha_{1}+\cdots+\alpha_{N}=-2.\label{eq:parameter-cond}
\end{equation}
Let $z\in\mathrm{Mat}(2,N)$ be written as $z=(z_{1},\dots,z_{N})$ with the
column vectors $z_{j}=\,^{t}(z_{1,j},z_{2,j})\in\mathbb{C}^{2}$. The Zariski
open subset $Z\subset\mathrm{Mat}(2,N)$, called the generic stratum with
respect to $H$, is defined by
\[
Z=\{z\in\mathrm{Mat}(2,N)\mid\det(z_{i},z_{j})\neq0,\quad1\leq\forall i\neq j\leq N\}.
\]
Any $z\in Z$ gives $N$ linear polynomials of $t=(t_{1},t_{2})$:
\[
tz_{j}=t_{1}z_{1,j}+t_{2}z_{2,j}=tz_{j},\quad1\leq j\leq N,
\]
where $t$ is considered as the homogeneous coordinates of the complex
projective space $\mathbb{P}^{1}$. The point of $\mathbb{P}^{1}$ with the homogeneous
coordinate $t$ will be denoted by $[t]$. Let $p_{j}(z)$ be the
zero of $tz_{j}$ in $\mathbb{P}^{1}$. Then, we see that $p_{1}(z),\dots,p_{N}(z)$
are distinct points in $\mathbb{P}^{1}$ by virtue of the condition $\det(z_{i},z_{j})\neq0,\;1\leq i\neq j\leq N$.
Identifying $tz$ with the diagonal matrix $\mathrm{diag}(tz_{1},\dots,tz_{N})\in H,$
define the Gelfand HGF by 
\begin{equation}
F(z,\alpha;C)=\int_{C(z)}\chi(tz;\alpha)\cdot\tau(t)=\int_{C(z)}\prod_{1\leq j\leq N}(tz_{j})^{\alpha_{j}}\cdot\tau(t),\label{eq:def-HGF}
\end{equation}
where $\tau(t)=t_{1}dt_{2}-t_{2}dt_{1}=t_{1}^{2}d(t_{2}/t_{1})$ and
$C(z)$ is a path connecting two points $p_{i}(z$) and $p_{j}(z)$
for example, which gives a cycle of the homology group $H_{1}^{lf}(X_{z};\mathcal{S})$
of locally finite chains of $X_{z}=\mathbb{P}^{1}\setminus\{p_{1}(z),\dots,p_{N}(z)\}$
with coefficients in the local system $\mathcal{S}$ determined by $\chi(tz,\alpha)\tau(t)$.
Put 
\[
\vec{u}=(1,u),\quad u=t_{2}/t_{1}.
\]
Then $u$ gives the affine coordinate in $U=\{[t]\in\mathbb{P}^{1}\mid t_{1}\neq0\}\simeq\mathbb{C}$.
By the condition (\ref{eq:parameter-cond}), the function $F$ defined
by (\ref{eq:def-HGF}) can be written as 
\[
F(z,\alpha;C)=\int_{C(z)}\chi(\vec{u}z;\alpha)du=\int_{C(z)}\prod_{1\leq j\leq N}(z_{1,j}+z_{2,j}u)^{\alpha_{j}}du.
\]
It is easy to check that we can define the action of $\mathrm{GL}(2) \times H$
on $Z$ by 
\begin{gather}
\mathrm{GL}(2) \times Z\times H\longrightarrow Z.\label{eq:action}\\
\qquad(g,z,h)\qquad\mapsto gzh\nonumber 
\end{gather}
Then we have the covariance property of the Gelfand HGF with respect
to the action $\mathrm{GL}(2) \curvearrowright Z\curvearrowleft H$ as follows.
See \cite{Gelfand}.
\begin{prop}
\label{prop:covariance}We have
\begin{align}
F(zh,\alpha;C) & =\chi(h,\alpha)F(z,\alpha;C),\qquad h\in H,\label{eq:cova-1}\\
F(gz,\alpha;C) & =(\det g)^{-1}F(z,\alpha;\tilde{C}),\qquad g\in\mathrm{GL}(2) ,\label{eq:cova-2}
\end{align}
where $\tilde{C}=\{\tilde{C}(z)\}$ is obtained from $C(z)$ as its
image by the projective transformation $\mathbb{P}^{1}\ni[t]\mapsto[s]:=[tg]\in\mathbb{P}^{1}$.
\end{prop}

Hereafter we write $F(z;\alpha)$ or $F(z)$ for $F(z,\alpha;C)$ for the
sake of simplicity. Let $\mathfrak{gl}(2)$ denote the Lie algebras of $\mathrm{GL}(2) $.
The following result is well known.
\begin{prop}
\label{prop:Gelfand-eq} The Gelfand HGF $F(z;\alpha)$ satisfies the
differential equations:

\begin{align}
\square_{p,q}F=(\partial_{1,p}\partial_{2,q}-\partial_{2,p}\partial_{1,q})F & =0,\quad1\leq p,q\leq N,\label{eq:gel-eq-1}\\
\left(\mathrm{Tr}(\,^{t}(zX)\partial)-\alpha(X)\right)F & =0,\quad X\in\mathfrak{h},\label{eq:gel-eq-2}\\
\left(\mathrm{Tr}(\,^{t}(Yz)\partial)+\mathrm{Tr}(Y)\right)F & =0,\quad Y\in\mathfrak{gl}(2),\label{eq:gel-eq-3}
\end{align}
where $\partial_{i,p}=\partial/\partial z_{i,p}$.
\end{prop}

The above system of differential equations will be called the \emph{Gelfand
hypergeometric system} (Gelfand HGS). The meaning of these differential
equations is as follows. For $X\in\mathfrak{g}$ and for a function $f$ on
$Z$, define the differential operator $L_{X}$ on $Z$ by 
\begin{equation}
L_{X}f:=\frac{d}{ds}f(z\exp sX)|_{s=0}=\mathrm{Tr}(\,^{t}(zX)\partial)\,f.\label{eq:gel-eq-4}
\end{equation}
Then we see that equation (\ref{eq:gel-eq-2}) is the infinitesimal
form of the property (\ref{eq:cova-1}). If we put $X=E_{p,p}$, the
$(p,p)$-th matrix unit, then $L_{p}:=L_{E_{p,p}}=z_{1,p}\partial_{1,p}+z_{2,p}\partial_{2,p}$
and (\ref{eq:gel-eq-2}) takes the form
\begin{equation}
L_{p}F(z;\alpha)=\alpha_{p}F(z;\alpha),\quad1\leq p\leq N,\label{eq:gel-eq-5}
\end{equation}
where $\alpha_{p}=\alpha(E_{p,p})$. Similarly we see that (\ref{eq:gel-eq-3})
is the infinitesimal form of (\ref{eq:cova-2}). The main body of
the Gelfand HGS is the system (\ref{eq:gel-eq-1}) which characterizes
the image of Radon transform. We will see in Section \ref{subsec:Reduction-HGS}
that a system of EPD equations arises from (\ref{eq:gel-eq-1}) in
a natural way.

\subsection{Contiguity relations of Gelfand's HGF}

We recall the facts about the contiguity operators and contiguity
relations of the Gelfand HGF. The contiguity operators play an important
role in establishing the Laplace sequence for the system of EPD equations
obtained from (\ref{eq:gel-eq-1}).

Recall that $\mathfrak{g}$ is the Lie algebra of $G=\mathrm{GL}(N) $ and $\mathfrak{h}$ is
the Cartan subalgebra of $\mathfrak{g}$ consisting of the diagonal matrices.
We consider the root space decomposition of $\mathfrak{g}$ with respect to
the adjoint action of $\mathfrak{h}$ on $\mathfrak{g}$; for $h\in\mathfrak{h}$, define $\mathrm{ad}\,h\in\mathrm{End}(\mathfrak{g})$
by 
\[
\mathrm{ad}\,h:\mathfrak{g}\ni X\mapsto(\mathrm{ad}\,h)X:=[h,X]=hX-Xh\in\mathfrak{g}.
\]
Since $\mathfrak{h}$ is an abelian Lie algebra, namely $[h,h']=0$ for any
$h,h'\in\mathfrak{h}$, $\{\mathrm{ad}\,h\mid h\in\mathfrak{h}\}$ forms a commuting subset
of $\mathrm{End}(\mathfrak{g})$. Since we have $(\mathrm{ad}\,h)X=[h,X]=\left((h_{i}-h_{j})X_{i,j}\right)_{1\leq i,j\leq N}$
for $h=\mathrm{diag}(h_{1},\dots,h_{N})$ and $X=(X_{i,j})$, we have the
decomposition of $\mathfrak{g}$ into the eigenspaces common for all $h\in\mathfrak{h}$
: 

\[
\mathfrak{g}=\mathfrak{h}\oplus\bigoplus_{i\neq j}\mathfrak{g}_{\epsilon_{i}-\epsilon_{j}},\quad\mathfrak{g}_{\epsilon_{i}-\epsilon_{j}}=\mathbb{C}\cdot E_{i,j},
\]
where $E_{i,j}$ is the $(i,j$)-th matrix unit, and $\epsilon_{i}\in\mathfrak{h}^{*}$
is defined by $\epsilon_{i}(h)=h_{i}$ for $h=\mathrm{diag}(h_{1},\dots,h_{N})$.
The subspace $\mathfrak{g}_{\epsilon_{i}-\epsilon_{j}}\subset\mathfrak{g}$ is the eigenspace
of $\mathrm{ad}\,h$ common for all $h\in\mathfrak{h}$ with the eigenvalue $(\epsilon_{i}-\epsilon_{j})(h)=h_{i}-h_{j}$,
and $\epsilon_{i}-\epsilon_{j}$ is called a root. 

The contiguity operators are constructed as follows. Let $L_{p,q}:=L_{E_{p,q}}$
be defined by (\ref{eq:gel-eq-4}) for $X=E_{p,q}$, then its explicit
form is
\begin{equation}
L_{p,q}=z_{1,p}\partial_{1,q}+z_{2,p}\partial_{2,q},\quad1\leq p,q\leq N.\label{eq:cont-1}
\end{equation}
The following is known and is easily shown. See \cite{Gelfand,Kimura-H-T,Horikawa,sasaki}.
\begin{prop}
\label{prop:Conti-1}For the Gelfand HGF $F(z;\alpha)$, the contiguity
relations are
\begin{equation}
L_{p,q}F(z;\alpha)=\alpha_{q}F(z;\alpha+e_{p}-e_{q}),\quad1\leq p\neq q\leq N,\label{eq:conti-2}
\end{equation}
where $e_{p}\in\mathbb{C}^{N}$ is the unit vector whose unique nonzero entry
$1$ locates at $p$-th position.
\end{prop}

\subsection{\label{subsec:Reduction-HGS}Reduction of Gelfand's HGS}

Suppose that $z=(z_{1},\dots,z_{N})\in Z$ satisfies $z_{2,j}\neq0$
for $1\leq j\leq N$. This condition can be understood as follows.
Each column vector $z_{j}$ defines a point $[z_{j}]$ in $\mathbb{P}^{1}$
considering $z_{j}$ as its homogeneous coordinate. Then the above
condition implies that $N$ points $[z_{1}],\dots,[z_{N}${]} belong
to the affine chart $\{[s]\in\mathbb{P}^{1}\mid s_{2}\neq0\}$, where $s=\,^{t}(s_{1},s_{2})$
is the homogeneous coordinates. Here we make a reduction of the system
(\ref{eq:gel-eq-1}) using the action of $H$. Consider the change
of variable $z\mapsto x=(x_{1},\dots,x_{N})$ defined by 
\[
\mathbf{x}=\left(\begin{array}{ccc}
x_{1} & \dots & x_{N}\\
1 & \dots & 1
\end{array}\right)=\left(\begin{array}{ccc}
z_{1,1} & \dots & z_{1,N}\\
z_{2,1} & \dots & z_{2,N}
\end{array}\right)\mathrm{diag}(z_{2,1}^{-1},\dots,z_{2,N}^{-1})
\]
and the change of unknown $F\mapsto\Phi$:
\[
F(z;\alpha)=V(z;\alpha)\Phi(x;\alpha),\quad V(z;\alpha)=\prod_{1\leq j\leq N}z_{2,j}^{\alpha_{j}}.
\]
This change of unknown is suggested by Proposition \ref{prop:covariance}.
In fact we have 
\begin{equation}
\Phi(x;\alpha)=F\left(\mathbf{x};\alpha\right),\label{eq:red-0}
\end{equation}
which is the restriction of the Gelfand HGF to the submanifold 
\[
\left\{ \left(\begin{array}{ccc}
x_{1} & \dots & x_{N}\\
1 & \dots & 1
\end{array}\right)\mathrm{Mat}(2,N)\mid x_{i}\neq x_{j},\;1\leq i\ne j\leq N\right\} \subset Z.
\]
Sometimes we write $\Phi(x)$ for $\Phi(x;\alpha)$ for the sake of brevity.
First we investigate how the condition (\ref{eq:gel-eq-5}) for $F(z)$
is translated to that for $\Phi(x)$. Put $\partial_{p}=\partial/\partial x_{p},\;1\leq p\leq N$.
Since $x_{p}=z_{1,p}/z_{2,p}$, when we apply $\partial_{1,p},\partial_{2,p}$
to a function of $x$, we have
\begin{align*}
\partial_{1,p} & =\frac{\partial x_{p}}{\partial z_{1,p}}\partial_{p}=\frac{1}{z_{2,p}}\partial_{p}=\frac{x_{p}}{z_{1,p}}\partial_{p},\\
\partial_{2,p} & =\frac{\partial x_{p}}{\partial z_{2,p}}\partial_{p}=-\frac{z_{1,p}}{z_{2,p}^{2}}\partial_{p}=-\frac{x_{p}}{z_{2,p}}\partial_{p},
\end{align*}
and hence 
\begin{equation}
z_{1,p}\partial_{1,p}=x_{p}\partial_{p},\quad z_{2,p}\partial_{2,p}=-x_{p}\partial_{p}.\label{eq:2-red-1}
\end{equation}

\begin{lem}
For the function $\Phi(x;\alpha)$ defined by (\ref{eq:red-0}), the condition
(\ref{eq:gel-eq-5}) becomes trivial.
\end{lem}

\begin{proof}
Recall that $L_{p}=z_{1,p}\partial_{1,p}+z_{2,p}\partial_{2,p}$ for $1\leq p\leq N$.
Since the variable $x_{p}:=z_{1,p}/z_{2,p}$ is invariant by the action
of $\mathbb{C}^{\times}$ defined by $(z_{1,p},z_{2,p})\mapsto(cz_{1,p},cz_{2,p})$
and $L_{p}$ is an infinitesimal expression of this action, we see
that $L_{p}\Phi(x)=0$. It follows that $L_{p}F(z)=L_{p}V(z)\cdot\Phi(x)+V(z)\cdot L_{p}\Phi(x)=L_{p}V(z)\cdot\Phi(x)$.
On the other hand we have $L_{p}V(z)=\alpha_{p}V(z)$ since $V(z)$ is
a homogeneous function of $z$ of degree $\alpha_{p}$. Hence the condition
(\ref{eq:gel-eq-5}) trivially holds and produces no condition on
$\Phi(x)$. 
\end{proof}
Next we consider the equations obtained from $\square_{p,q}F=0$.
\begin{prop}
\label{prop:red-EPD}The equations $\square_{p,q}F=0,\;1\leq p\neq q\leq N$,
give the system of EPD equations for $\Phi(x;\alpha)$:
\[
\left\{ (x_{p}-x_{q})\partial_{p}\partial_{q}+\alpha_{q}\partial_{p}-\alpha_{p}\partial_{q}\right\} \Phi(x;\alpha)=0,\quad1\leq p,q\leq N.
\]
\end{prop}

\begin{proof}
For $F=V(z)\Phi(x)=(\prod_{j=1}^{N}z_{2,j}^{\alpha_{j}})\Phi(x)$, taking
account of (\ref{eq:2-red-1}), we have 
\begin{align}
\square_{p,q}F & =(\partial_{1,p}\partial_{2,q}-\partial_{2,p}\partial_{1,q})V(z)\Phi(x)\nonumber \\
 & =V(z)\left\{ \left(\frac{\alpha_{q}}{z_{2,q}}\partial_{1,p}-\frac{\alpha_{p}}{z_{2,p}}\partial_{1,q}\right)+(\partial_{1,p}\partial_{2,q}-\partial_{2,p}\partial_{1,q})\right\} \Phi(x)\nonumber \\
 & =V(z)\left\{ \left(\frac{\alpha_{q}}{z_{2,p}z_{2,q}}\partial_{p}-\frac{\alpha_{p}}{z_{2,p}z_{2,q}}\partial_{q}\right)+(\partial_{1,p}\partial_{2,q}-\partial_{2,p}\partial_{1,q})\right\} \Phi(x).\label{eq:2-red-2}
\end{align}
The second order differential operator in the last line of (\ref{eq:2-red-2})
acts on $\Phi(x)$ as 
\begin{align}
\partial_{1,p}\partial_{2,q}-\partial_{2,p}\partial_{1,q} & =\left(\frac{x_{p}}{z_{1,p}}\partial_{p}\right)\left(-\frac{x_{q}}{z_{2,q}}\partial_{q}\right)-\left(-\frac{x_{p}}{z_{2,p}}\partial_{p}\right)\left(\frac{x_{q}}{z_{1,q}}\partial_{q}\right)\label{eq:2-red-3}\\
 & =\frac{1}{z_{2,p}z_{2,q}}\left(-x_{q}+x_{p}\right)\partial_{p}\partial_{q}.\nonumber 
\end{align}
Multiplying the both sides (\ref{eq:2-red-2}) by $z_{2,p}z_{2.q}$
and using (\ref{eq:2-red-3}), we have from $\square_{p,q}F=0$ the
EPD equation $\left\{ (x_{p}-x_{q})\partial_{p}\partial_{q}+\alpha_{q}\partial_{p}-\alpha_{p}\partial_{q}\right\} \Phi(x)=0$.
\end{proof}

\subsection{Reduction of the contiguity relations}

Let us translate the contiguity relations (\ref{eq:conti-2}) for
$F(z;\alpha)$ to those for $\Phi(x;\alpha)$. To this end, we rewrite the
operators $L_{p,q}=z_{1,p}\partial_{1,q}+z_{2,p}\partial_{2,q}$ for $F(z;\alpha)$
to those for $\Phi(x;\alpha)$.
\begin{lem}
\label{lem:Conti-3} If $p\neq q$, the differential operator $L_{p,q}$
acts on a function of $x$ as 
\[
L_{p,q}=\frac{z_{2,p}}{z_{2,q}}(x_{p}-x_{q})\partial_{q}.
\]
\end{lem}

\begin{proof}
For a function $f$ of $x$, we have
\begin{align*}
L_{p,q} & f=(z_{1,p}\partial_{1,q}+z_{2,p}\partial_{2,q})f=z_{1,p}\left(\frac{x_{q}}{z_{1,q}}\right)\partial_{q}f+z_{2,p}\left(-\frac{x_{q}}{z_{2,q}}\right)\partial_{q}f\\
 & =\frac{z_{2,p}}{z_{2,q}}\left(x_{p}-x_{q}\right)\partial_{q}f.
\end{align*}
 
\end{proof}
\begin{prop}
\label{prop:Conti-4}The contiguity relations for $\Phi(x;\alpha)$ are
given by
\begin{equation}
\mathcal{L}_{p,q}\Phi(x;\alpha)=\alpha_{q}\Phi(x;\alpha+e_{p}-e_{q}),\quad1\leq p\neq q\leq N,\label{eq:2-red-4}
\end{equation}
with the differential operators $\mathcal{L}_{p,q}:=(x_{p}-x_{q})\partial_{q}+\alpha_{q}$.
\end{prop}

\begin{proof}
In Proposition \ref{prop:Conti-1}, we gave the contiguity relations
for $F(z;\alpha)$:
\begin{equation}
L_{p,q}\cdot F(z;\alpha)=\alpha_{q}F(z;\alpha+e_{p}-e_{q}),\label{eq:2-red-5}
\end{equation}
where $L_{p,q}=z_{1,p}\partial_{1,q}+z_{2,p}\partial_{2,q}$. By Lemma \ref{lem:Conti-3},
$L_{p,q}=(z_{2,p}/z_{2,q})(x_{p}-x_{q})\partial_{q}$ when it is applied
to a function of $x$. Putting $F(z;\alpha)=V(z)\Phi(x;\alpha)$, $V(z)=\prod_{1\leq j\leq N}z_{2,j}^{\alpha_{j}}$,
in the left hand side of (\ref{eq:2-red-5}) and noting $L_{p,q}V(z)=(z_{2,p}/z_{2,q})\alpha_{q}V(z)$,
we have 
\begin{align*}
L_{p,q}F(z;\alpha) & =L_{p,q}V(z)\cdot\Phi(x;\alpha)+V(z)\cdot L_{p,q}\Phi(x;\alpha)\\
 & =\frac{z_{2,p}}{z_{2,q}}\alpha_{q}V(z)\cdot\Phi(x;\alpha)+V(z)\cdot\frac{z_{2,p}}{z_{2,q}}(x_{p}-x_{q})\partial_{q}\Phi(x;\alpha)\\
 & =\frac{z_{2,p}}{z_{2,q}}V(z)\left((x_{p}-x_{q})\partial_{q}+\alpha_{q}\right)\Phi(x;\alpha).
\end{align*}
On the other hand $F(z;\alpha+e_{p}-e_{q})=(z_{2,p}/z_{2,q})V(z)\Phi(x;\alpha+e_{p}-e_{q})$.
Then, from (\ref{eq:2-red-5}) we have
\[
\left\{ (x_{p}-x_{q})\partial_{q}+\alpha_{q}\right\} \Phi(x;\alpha)=\alpha_{q}\Phi(x;\alpha+e_{p}-e_{q}).
\]
\end{proof}

\section{Gelfand HGF as a solution of the  2dTHE}

As is seen in Section \ref{subsec:Reduction-HGS}, we obtained the
system of EPD equations
\begin{equation}
\mathcal{M}(\alpha):M_{p,q}(\alpha)u=\left\{ \partial_{p}\partial_{q}+\frac{\alpha_{q}}{x_{p}-x_{q}}\partial_{p}+\frac{\alpha_{p}}{x_{q}-x_{p}}\partial_{q}\right\} u=0,\;1\leq p\neq q\leq N\label{eq:gel-toda-1}
\end{equation}
from the system (\ref{eq:gel-eq-1}) as a result of reduction by the
group action $Z\curvearrowleft H$ and the covariance property given
in Proposition \ref{prop:covariance}. Note that the Gelfand HGF $F(z;\alpha)$
is characterized by the Gelfand HGS (\ref{eq:gel-eq-1}), (\ref{eq:gel-eq-2})
and (\ref{eq:gel-eq-3}). Following the process of reduction, we have
seen that the system $\mathcal{M}(\alpha)$ has a solution $\Phi(x;\alpha)$ which
is related to the Gelfand HGF $F(z;\alpha)$ by 
\begin{equation}
F(z;\alpha)=\left(\prod_{1\leq j\leq N}z_{2,j}^{\alpha_{j}}\right)\Phi(x;\alpha).\label{eq:gel-toda-1-1}
\end{equation}
By the same reduction, we obtained the operators
\[
L_{p,q}(\alpha)=(x_{p}-x_{q})\partial_{q}+\alpha_{q},\quad1\leq p\neq q\leq N
\]
from the contiguity operators of the Gelfand HGF. These operators
describe the contiguity relations of $\Phi(x;\alpha)$ as we have seen
in Proposition \ref{prop:Conti-4}.

\subsection{\label{subsec:Compatibility}Generator of the ideal for the system
$\protect\mathcal{M}(\protect\alpha)$}

Let $\mathcal{R}=\mathbb{C}[x,\prod_{a<b}(x_{a}-x_{b})^{-1}]\langle\partial_{1},\dots,\partial_{N}\rangle$
be the ring of differential operators with coefficients in the ring
$\mathbb{C}[x,\prod_{a<b}(x_{a}-x_{b})^{-1}]$, where $\mathbb{C}[x,\prod_{a<b}(x_{a}-x_{b})^{-1}]$
is the localization of the polynomial ring $\mathbb{C}[x]$ by the polynomial
$\prod_{a<b}(x_{a}-x_{b})$. Let $\mathcal{I}(\alpha)$ denote the left ideal
of $\mathcal{R}$ generated by EPD operators $\{M_{i,j}(\alpha)\}_{1\leq i\ne j\leq N}$.
We show the following fact which says that we can take a particular
generator of $\mathcal{I}(\alpha)$ consisting of $N-1$ operators. It will be
seen in Lemma \ref{lem:Gel-toda-2} that it corresponds to the set
of simple roots for $\mathfrak{gl}(N)$. 
\begin{prop}
\label{prop:EPD-1}For any distinct $1\leq i,j,k\leq N$, we have
the identity:
\begin{align}
 & S(M_{i,j}(\alpha),M_{j,k}(\alpha)):=\partial_{k}M_{i,j}(\alpha)-\partial_{i}M_{j,k}(\alpha)\label{eq:epd-1}\\
 & =-\alpha_{j}\left(\frac{x_{k}-x_{i}}{(x_{i}-x_{j})(x_{j}-x_{k})}\right)M_{i,k}(\alpha)-\frac{\alpha_{k}}{x_{j}-x_{k}}M_{i,j}(\alpha)-\frac{\alpha_{i}}{x_{i}-x_{j}}M_{j,k}(\alpha).\nonumber 
\end{align}
Under the condition $\alpha_{j}\neq0$ for $1\leq\forall j\leq N$, the
ideal $\mathcal{I}(\alpha)$ has a generator $\{M_{i,i+1}(\alpha)\}_{1\leq i\leq N-1}$. 
\end{prop}

\begin{proof}
We write $M_{i,j}(\alpha)$ as $M_{i,j}$ and we compute the left hand
side of (\ref{eq:epd-1}).
\begin{align*}
 & \partial_{k}M_{i,j}-\partial_{i}M_{j,k}\\
 & =\partial_{k}\left(\frac{\alpha_{j}}{x_{i}-x_{j}}\partial_{i}+\frac{\alpha_{i}}{x_{j}-x_{i}}\partial_{j}\right)-\partial_{i}\left(\frac{\alpha_{k}}{x_{j}-x_{k}}\partial_{j}+\frac{\alpha_{j}}{x_{k}-x_{j}}\partial_{k}\right)\\
 & =\alpha_{j}\left(\frac{1}{x_{i}-x_{j}}-\frac{1}{x_{k}-x_{j}}\right)\partial_{i}\partial_{k}-\frac{\alpha_{k}}{x_{j}-x_{k}}\partial_{i}\partial_{j}+\frac{\alpha_{i}}{x_{j}-x_{i}}\partial_{j}\partial_{k}\\
 & =\alpha_{j}\left(\frac{x_{k}-x_{i}}{(x_{i}-x_{j})(x_{k}-x_{j})}\right)M_{i,k}-\frac{\alpha_{k}}{x_{j}-x_{k}}M_{i,j}+\frac{\alpha_{i}}{x_{j}-x_{i}}M_{j,k}+R,
\end{align*}
Then it is immediate to see that $R=0$. Hence (\ref{eq:epd-1}) is
established. The second assertion may be obvious. In fact, to obtain
$M_{1,3}(\alpha)$ for example, we choose the indices $(1,2,3)$ as $(i,j,k)$
in (\ref{eq:epd-1}). Then we have 
\begin{multline*}
\alpha_{2}M_{1,3}(\alpha)=\frac{(x_{1}-x_{2})(x_{2}-x_{3})}{x_{1}-x_{3}}\Bigl(S(M_{1,2}(\alpha),M_{2,3}(\alpha)) \\
+\frac{\alpha_{3}}{x_{2}-x_{3}}M_{1,2}(\alpha)+\frac{\alpha_{1}}{x_{1}-x_{2}}M_{2,3}(\alpha)\Bigr)
\end{multline*}
and the right hand side is given by using only $M_{1,2}(\alpha),M_{2,3}(\alpha)$.
\end{proof}
\begin{rem}
$S(M_{i,j}(\alpha),M_{j,k}(\alpha))$ in Proposition \ref{prop:EPD-1} is
an S-pair of $M_{i,j}(\alpha)$ and $M_{j,k}(\alpha)$ in the ring $\mathcal{R}$
with an appropriate ordering which is used in the theory of Gr\"obner
basis for the ring of differential operators. 
\end{rem}

\subsection{$SL(2,\protect\mathbb{C})$ action on the solution space of $\protect\mathcal{M}(\protect\alpha)$}

In this section we consider the $SL(2,\mathbb{C})$ action on solutions of
$\mathcal{M}(\alpha)$. 
\begin{prop}
For a solution $u(x)$ of $\mathcal{M}(\alpha)$ and $g\in SL(2,\mathbb{C})$, define
$\tilde{u}(x)$ by 
\[
\tilde{u}(x):=\prod_{1\leq k\leq N}(cx_{k}+d)^{\alpha_{k}}\cdot u\left(\frac{ax_{1}+b}{cx_{1}+d},\cdots,\frac{ax_{N}+b}{cx_{N}+d}\right),\;g=\left(\begin{array}{cc}
a & b\\
c & d
\end{array}\right).
\]
Then $\tilde{u}(x)$ is also a solution of $\mathcal{M}(\alpha)$.
\end{prop}

Before giving the proof of the proposition, we explain a motivation
to consider the transformation $u\mapsto\tilde{u}$ in the proposition.
We know that the system $\mathcal{M}(\alpha)$ has a solution $\Phi(x;\alpha)$ which
is defined by $\Phi(x;\alpha):=F(\mathbf{x};\alpha)$ by restricting the Gelfand
HGF $F(z;\alpha)$ to 
\[
\left\{ \mathbf{x}=\left(\begin{array}{ccc}
x_{1} & \dots & x_{N}\\
1 & \dots & 1
\end{array}\right)\right\} \subset Z.
\]
Take $g\in SL(2,\mathbb{C})$ as above and consider the transformation
\[
\mathbf{x}\mapsto\mathbf{x}':=g\mathbf{x} h^{-1}=\left(\begin{array}{ccc}
\frac{ax_{1}+b}{cx_{1}+d} & \dots & \frac{ax_{N}+b}{cx_{N}+d}\\
1 & \dots & 1
\end{array}\right)
\]
with $h=\mathrm{diag}(cx_{1}+d,\dots,cx_{N}+d)$. Then we have
\begin{align}
\Phi(x;\alpha) & =F(\mathbf{x};\alpha)=F(g^{-1}\mathbf{x}'h;\alpha)=\det g\cdot\chi(h;\alpha)F(\mathbf{x}';\alpha)\label{eq:sl-2}\\
 & =\prod_{1\leq k\leq N}(cx_{k}+d)^{\alpha_{k}}\cdot\Phi\left(\frac{ax_{1}+b}{cx_{1}+d},\cdots,\frac{ax_{N}+b}{cx_{N}+d};\alpha\right).\nonumber 
\end{align}
Since $\Phi(x;\alpha)$ is a solution of the system $\mathcal{M}(\alpha)$, the right
hand side of (\ref{eq:sl-2}) also satisfies $\mathcal{M}(\alpha)$. This fact
motivates to consider the transformation $u\mapsto\tilde{u}$ in the
proposition. Now we give the proof.
\begin{proof}
We have to show $M_{i,j}(\alpha)\tilde{u}(x)=0$ for any $i\ne j$. Noting
$M_{i,j}(\alpha)$ contains the derivations $\partial_{i},\partial_{j}$ only and
taking into account the form of transformation $x_{k}\mapsto(ax_{k}+b)/(cx_{k}+d)$,
we can regard other variables $x_{a}\;(a\neq i,j)$ as fixed constants.
Hence the proof is reduced to the $2$ variables case; let $u(x,y)$
is a solution of single EPD equation
\[
Mu=\left(\partial_{x}\partial_{y}+\frac{\beta}{x-y}\partial_{x}+\frac{\alpha}{y-x}\partial_{y}\right)u=0
\]
and let 
\[
\tilde{u}(x,y)=A(x,y)u\left(\frac{ax+b}{cx+d},\frac{ay+b}{cy+d}\right),\;A(x,y)=(cx+d)^{\alpha}(cy+d)^{\beta}.
\]
Put $X=(ax+b)/(cx+d),Y=(ay+b)/(cy+d)$. Then 
\begin{align*}
\partial_{x}\tilde{u} & =\frac{\alpha c}{cx+d}A\cdot u(X,Y)+\frac{1}{(cx+d)^{2}}A\cdot u_{x}(X,Y),\\
\partial_{y}\tilde{u} & =\frac{\beta c}{cy+d}A\cdot u(X,Y)+\frac{1}{(cy+d)^{2}}A\cdot u_{y}(X,Y),\\
\partial_{x}\partial_{y}\tilde{u} & =\frac{A}{(cx+d)^{2}(cy+d)^{2}}\left\{ u_{xy}(X,Y)+\beta c(cy+d)u_{x}(X,Y)\right.\\
 & \left.+\alpha c(cx+d)u_{y}(X,Y)+\alpha\beta c^{2}(cx+d)(cy+d)u(X,Y)\right\} .
\end{align*}
Then multiplying $M\tilde{u}$ by $(cx+d)^{2}(cy+d)^{2}/A$ and using
the above expressions, we have 
\begin{align*}
M\tilde{u} & \to u_{xy}(X,Y)+\beta\frac{(cx+d)(cy+d)}{x-y}u_{x}(X,Y)\\
 & +\alpha\frac{(cx+d)(cy+d)}{y-x}u_{y}(X,Y)\\
 & =u_{xy}(X,Y)+\frac{\beta}{X-Y}u_{x}(X,Y)+\frac{\alpha}{Y-X}u_{y}(X,Y)\\
 & =0.
\end{align*}
 This proves the proposition.
\end{proof}

\subsection{Contiguity for the system $\protect\mathcal{M}(\protect\alpha)$}

Let $\mathcal{S}(\alpha)$ denote the space of holomorphic solutions of the system
$\mathcal{M}(\alpha)$ in some simply connected domain $\Omega'\subset\mathbb{C}^{N}\setminus\cup_{i\neq j}\{x_{i}=x_{j}\}$.
We also use $\mathcal{S}_{p,q}(\alpha)$ to denote the set of holomorphic solutions
of the single EPD equation $M_{p,q}(\alpha)u=0$. Then $\mathcal{S}(\alpha)=\cap_{p\neq q}\mathcal{S}_{p,q}(\alpha)$.
In Proposition \ref{prop:Conti-4}, we gave the contiguity relation
for the Gelfand HGF $\Phi(x;\alpha)$, where the differential operator
$\mathcal{L}_{p,q}(\alpha)$ is used. In this section $\mathcal{L}_{p,q}(\alpha)$ will be
denoted as $L_{p,q}(\alpha)$, namely,
\[
L_{p,q}(\alpha)=(x_{p}-x_{q})\partial_{q}+\alpha_{q},\quad1\leq p\neq q\leq N.
\]
It is natural to expect that $L_{p,q}(\alpha)$ defines a linear map $L_{p,q}(\alpha):\mathcal{S}_{p,q}(\alpha)\to\mathcal{S}_{p,q}(\alpha+e_{p}-e_{q})$.
This is correct and will be shown in Lemma \ref{lem:Gel-toda-4}.
But we can show more. For any fixed pair $(i,j),1\leq i\neq j\leq N$,
we can show that $L_{i,j}(\alpha)$ defines a linear map $L_{i,j}(\alpha):\mathcal{S}(\alpha)\to\mathcal{S}(\alpha+e_{i}-e_{j})$.
From now on we fix a pair $(i,j)$ in this section. Then we can show
the following.
\begin{prop}
\label{prop:Gel-toda-1}If $u\in\mathcal{S}(\alpha)$, then $L_{i,j}(\alpha)u\in\mathcal{S}(\alpha+e_{i}-e_{j})$.
Under the condition $(\alpha_{i}+1)\alpha_{j}\neq0$, the linear map $\mathcal{S}(\alpha)\ni u\mapsto L_{i,j}(\alpha)u\in\mathcal{S}(\alpha+e_{i}-e_{j})$
is an isomorphism. The inverse map is given by
\[
\frac{1}{(\alpha_{i}+1)\alpha_{j}}L_{j,i}(\alpha+e_{i}-e_{j}).
\]
\end{prop}

To show this proposition, we prepare several lemmas. 
\begin{lem}
\label{lem:Gel-toda-2}For any $1\leq p\neq q\leq N$, we have 
\begin{equation}
L_{q,p}(\alpha+e_{p}-e_{q})L_{p,q}(\alpha)=-(x_{p}-x_{q})^{2}M_{p,q}(\alpha)+(\alpha_{p}+1)\alpha_{q}.\label{eq:gel-toda-2}
\end{equation}
\end{lem}

\begin{proof}
Let us compute the left hand side.

\begin{align*}
 & L_{q,p}(\alpha+e_{p}-e_{q})L_{p,q}(\alpha)\\
 & =\left((x_{q}-x_{p})\partial_{p}+(\alpha_{p}+1)\right)\left((x_{p}-x_{q})\partial_{q}+\alpha_{q}\right)\\
 & =(x_{q}-x_{p})\partial_{p}\cdot(x_{p}-x_{q})\partial_{q}+\alpha_{q}(x_{q}-x_{p})\partial_{p}+(\alpha_{p}+1)(x_{p}-x_{q})\partial_{q}+(\alpha_{p}+1)\alpha_{q}\\
 & =-(x_{p}-x_{q})^{2}\partial_{p}\partial_{q}+\alpha_{q}(x_{q}-x_{p})\partial_{p}+\alpha_{p}(x_{p}-x_{q})\partial_{q}+(\alpha_{p}+1)\alpha_{q}\\
 & =-(x_{p}-x_{q})^{2}\left\{ \partial_{p}\partial_{q}+\frac{\alpha_{q}}{x_{p}-x_{q}}\partial_{p}+\frac{\alpha_{p}}{x_{q}-x_{p}}\partial_{q}\right\} +(\alpha_{p}+1)\alpha_{q}\\
 & =-(x_{p}-x_{q})^{2}M_{p,q}(\alpha)+(\alpha_{p}+1)\alpha_{q}.
\end{align*}
Thus the lemma is proved.
\end{proof}
\begin{lem}
\label{lem:Gel-toda-3}For any $1\leq p\neq q\leq N$, we have 
\begin{equation}
(x_{p}-x_{q})^{2}M_{p,q}(\alpha+e_{p}-e_{q})L_{p,q}(\alpha)=L_{p,q}(\alpha)\cdot(x_{p}-x_{q})^{2}M_{p,q}(\alpha).\label{eq:gel-toda-4}
\end{equation}
\end{lem}

\begin{proof}
From (\ref{eq:gel-toda-2}) we can obtain
\begin{equation}
L_{p,q}(\alpha)L_{q,p}(\alpha+e_{p}-e_{q})=-(x_{q}-x_{p})^{2}M_{p,q}(\alpha+e_{p}-e_{q})+(\alpha_{p}+1)\alpha_{q}\label{eq:gel-toda-3}
\end{equation}
Indeed, we exchange the index $p\leftrightarrow q$ in (\ref{eq:gel-toda-2})
and note that $M_{p,q}(\alpha)=M_{q,p}(\alpha)$. Then we have 
\[
L_{p,q}(\alpha-e_{p}+e_{q})L_{q,p}(\alpha)=-(x_{q}-x_{p})^{2}M_{p,q}(\alpha)+(\alpha_{q}+1)\alpha_{p}.
\]
In this expression, we make a replacement $\alpha\to\alpha+e_{p}-e_{q}$ and
obtain (\ref{eq:gel-toda-3}). Using this identity, we have 
\begin{align*}
 & (x_{p}-x_{q})^{2}M_{p,q}(\alpha+e_{p}-e_{q})L_{p,q}(\alpha)\\
 & =\left\{ (\alpha_{p}+1)\alpha_{q}-L_{p,q}(\alpha)L_{q,p}(\alpha+e_{p}-e_{q})\right\} L_{p,q}(\alpha)\\
 & =(\alpha_{p}+1)\alpha_{q}L_{p,q}(\alpha)-L_{p,q}(\alpha)\left(L_{q,p}(\alpha+e_{p}-e_{q})L_{p,q}(\alpha)\right)\\
 & =(\alpha_{p}+1)\alpha_{q}L_{p,q}(\alpha)-L_{p,q}(\alpha)\left(-(x_{p}-x_{q})^{2}M_{p,q}(\alpha)+(\alpha_{p}+1)\alpha_{q}\right)\\
 & =L_{p,q}(\alpha)\cdot(x_{p}-x_{q})^{2}M_{p,q}(\alpha).
\end{align*}
 At the third equality, we used (\ref{eq:gel-toda-2}).
\end{proof}
Note that the indices $i$ and $j$ are fixed. To prove Proposition
\ref{prop:Gel-toda-1}, we check the assertion case by case. We want
to know under what condition $u\in\mathcal{S}_{p,q}(\alpha)$ is sent to $\mathcal{S}_{p,q}(\alpha+e_{i}-e_{j})$
by the operator $L_{i,j}(\alpha)$. 
\begin{lem}
\label{lem:Gel-toda-4}If $u\in\mathcal{S}_{i,j}(\alpha)$, then $L_{i,j}(\alpha)u\in\mathcal{S}_{i,j}(\alpha+e_{i}-e_{j})$. 
\end{lem}

\begin{proof}
For $u\in\mathcal{S}_{i,j}(\alpha)$, we show that $v=L_{i,j}(\alpha)u$ satisfies
$M_{i,j}(\alpha+e_{i}-e_{j})v=0$. In fact, by virtue of Lemma \ref{lem:Gel-toda-3},
we have 
\begin{align*}
(x_{i}-x_{j})^{2}M_{i,j}(\alpha+e_{i}-e_{j})v & =(x_{i}-x_{j})^{2}M_{i,j}(\alpha+e_{i}-e_{j})L_{i,j}(\alpha)u\\
 & =L_{i,j}(\alpha)\cdot(x_{i}-x_{j})^{2}M_{i,j}(\alpha)u\\
 & =0
\end{align*}
since $M_{i,j}(\alpha)u=0$ by the assumption. This proves the lemma.
\end{proof}
\begin{lem}
\label{lem:Gel-toda-5}In the case $\{i,j\}\cap\{p,q\}=\emptyset$,
the correspondence $u\mapsto L_{i,j}(\alpha)u$ gives a linear map $\mathcal{S}_{p,q}(\alpha)\to\mathcal{S}_{p,q}(\alpha+e_{i}-e_{j})$. 
\end{lem}

\begin{proof}
Since $\{i,j\}\cap\{p,q\}=\emptyset$, $M_{p,q}(\alpha+e_{i}-e_{j})=M_{p,q}(\alpha)$
and hence $\mathcal{S}_{p,q}(\alpha+e_{i}-e_{j})=\mathcal{S}_{p,q}(\alpha)$. Note that
\[
[M_{p,q}(\alpha),L_{i,j}(\alpha)]=\left[\partial_{p}\partial_{q}+\frac{\alpha_{q}}{x_{p}-x_{q}}\partial_{p}+\frac{\alpha_{p}}{x_{q}-x_{p}}\partial_{q},(x_{i}-x_{j})\partial_{j}+\alpha_{j}\right]=0.
\]
Then, for $u\in\mathcal{S}_{p,q}(\alpha)$, $v:=L_{i,j}(\alpha)u$ satisfies
\[
M_{p,q}(\alpha+e_{i}-e_{j})v=M_{p,q}(\alpha)L_{i,j}(\alpha)u=L_{i,j}(\alpha)M_{p,q}(\alpha)u=0.
\]
This implies $v\in\mathcal{S}_{p,q}(\alpha+e_{i}-e_{j})$. 
\end{proof}
Next we treat the case $\#(\{i,j\}\cap\{p,q\})=1$. Then $i\in\{p,q\}$
or $j\in\{p,q\}$. Noting $\mathcal{S}_{p,q}(\alpha)=\mathcal{S}_{q,p}(\alpha)$, we may assume
that $p=i$ and $q\ne i,j$ in the case $i\in\{p,q\}$, and $p=j$
and $q\ne i,j$ in the case $j\in\{p,q\}$. Let $\mathcal{R}$ be the ring
of differential operators defined in Section \ref{subsec:Compatibility}.
For $P\in\mathcal{R}$, we denote by $\mathcal{R}\cdot P$ the left ideal of $\mathcal{R}$
generated by $P$. 
\begin{lem}
\label{lem:Gel-toda-6}For any distinct $1\leq p,q,r\leq N$, we have
\begin{align}
L_{p,q}(\alpha+e_{q})L_{q,r}(\alpha) & \equiv (\alpha_{q}+1)L_{p,r}(\alpha)\;\text{modulo }\;\mathcal{R}\cdot M_{q,r}(\alpha).\label{eq:gel-toda-6}\\
L_{q,r}(\alpha)L_{p,q}(\alpha) & \equiv\alpha_{q}L_{p,r}(\alpha)\;\text{modulo}\;\mathcal{R}\cdot M_{q,r}(\alpha).\label{eq:gel-toda-7}
\end{align}
\end{lem}

\begin{proof}
We show (\ref{eq:gel-toda-6}). Noting
\[
M_{q,r}(\alpha)=\partial_{q}\partial_{r}+\frac{\alpha_{r}}{x_{q}-x_{r}}\partial_{q}+\frac{\alpha_{q}}{x_{r}-x_{q}}\partial_{r},
\]
we have
\begin{align*}
 & L_{p,q}(\alpha+e_{q})L_{q,r}(\alpha)\\
 & =\left((x_{p}-x_{q})\partial_{q}+\alpha_{q}+1\right)\left((x_{q}-x_{r})\partial_{r}+\alpha_{r}\right)\\
 & =(x_{p}-x_{q})(x_{q}-x_{r})\partial_{q}\partial_{r}+(x_{p}-x_{q})\partial_{r}+(\alpha_{q}+1)(x_{q}-x_{r})\partial_{r}\\
 & \quad+\alpha_{r}(x_{p}-x_{q})\partial_{q}+(\alpha_{q}+1)\alpha_{r}\\
 & \equiv(x_{p}-x_{q})(x_{q}-x_{r})\left\{ -\frac{\alpha_{r}}{x_{q}-x_{r}}\partial_{q}-\frac{\alpha_{q}}{x_{r}-x_{q}}\partial_{r}\right\} +(x_{p}-x_{q})\partial_{r}\\
 & \quad+(\alpha_{q}+1)(x_{q}-x_{r})\partial_{r}+\alpha_{r}(x_{p}-x_{q})\partial_{q}+(\alpha_{q}+1)\alpha_{r}\\
 & =(\alpha_{q}+1)L_{p,r}(\alpha).
\end{align*}
The formula (\ref{eq:gel-toda-7}) is shown in a similar way. 
\end{proof}
Using Lemma \ref{lem:Gel-toda-6}, we show the following, which will
complete the proof of Proposition \ref{prop:Gel-toda-1}.
\begin{lem}
Assume that $1\leq i,j,q\leq N$ are distinct. If $u\in\mathcal{S}_{i,j}(\alpha)\cap\mathcal{S}_{i,q}(\alpha)\cap\mathcal{S}_{j,q}(\alpha)$,
then $v=L_{i,j}(\alpha)u$ belongs to $\mathcal{S}_{i,j}(\alpha+e_{i}-e_{j})\cap\mathcal{S}_{j,q}(\alpha+e_{i}-e_{j})\cap\mathcal{S}_{i,q}(\alpha+e_{i}-e_{j})$.
\end{lem}

\begin{proof}
The fact $v\in\mathcal{S}_{i,j}(\alpha+e_{i}-e_{j})$ is already shown in Lemma
\ref{lem:Gel-toda-4}. We shall show $v\in\mathcal{S}_{i,q}(\alpha+e_{i}-e_{j})$.
Note that $M_{i,q}(\alpha+e_{i}-e_{j})=M_{i,q}(\alpha+e_{i})$ and hence the
equality $\mathcal{S}_{i,q}(\alpha+e_{i}-e_{j})=\mathcal{S}_{i,q}(\alpha+e_{i})$ holds. Put
$\beta=\alpha+e_{q}$. Then, using (\ref{eq:gel-toda-3}) replacing $\alpha$
with $\beta$, we have 
\begin{align*}
 & (x_{i}-x_{q})^{2}M_{i,q}(\alpha+e_{i})L_{i,j}(\alpha)\\
 & =(x_{i}-x_{q})^{2}M_{i,q}(\beta+e_{i}-e_{q})L_{i,j}(\alpha)\\
 & =\left\{ (\beta_{i}+1)\beta_{q}-L_{q,i}(\beta)L_{i,q}(\beta+e_{i}-e_{q})\right\} L_{i,j}(\alpha)\\
 & =(\beta_{i}+1)\beta_{q}L_{i,j}(\alpha)-L_{i,q}(\beta)\left\{ L_{q,i}(\beta+e_{i}-e_{q})L_{i,j}(\alpha)\right\} .
\end{align*}
By applying (\ref{eq:gel-toda-6}) of Lemma \ref{lem:Gel-toda-6},
the second term of the last line above is written as 
\begin{align*}
 & L_{i,q}(\beta)\Bigl\{ L_{q,i}(\beta+e_{i}-e_{q})L_{i,j}(\alpha)\Bigr\} \\
 & =L_{i,q}(\alpha+e_{q})\Bigl\{ L_{q,i}(\alpha+e_{i})L_{i,j}(\alpha)\Bigr\} \\
 & \equiv (\alpha_{i}+1)L_{i,q}(\alpha+e_{q})L_{q,j}(\alpha)\quad\text{modulo}\;\mathcal{R}\cdot M_{i,j}(\alpha)\\
 & \equiv (\alpha_{q}+1)(\alpha_{i}+1)L_{i,j}(\alpha)\quad\text{modulo}\;\mathcal{R}\cdot M_{q,j}(\alpha).
\end{align*}
Thus we have 
\begin{align*}
(x_{i}-x_{q})^{2}M_{i,q}(\alpha+e_{i})L_{i,j}(\alpha) & \equiv\beta_{q}(\beta_{i}+1)L_{i,j}(\alpha)-(\alpha_{q}+1)(\alpha_{i}+1)L_{i,j}(\alpha)\\
 & =0\quad\text{modulo}\;\mathcal{R}\cdot M_{i,j}(\alpha)+\mathcal{R}\cdot M_{q,j}(\alpha)
\end{align*}
since $\beta_{i}=\alpha_{i},\beta_{q}=\alpha_{q}+1$. Then, for $u\in\mathcal{S}_{i,j}(\alpha)\cap\mathcal{S}_{i,q}(\alpha)\cap\mathcal{S}_{j,q}(\alpha)$,
we have 
\[
(x_{i}-x_{q})^{2}M_{i,q}(\alpha+e_{i}-e_{j})v=(x_{i}-x_{q})^{2}M_{i,q}(\alpha+e_{i})L_{i,j}(\alpha)u=0.
\]
This implies $v\in\mathcal{S}_{i,q}(\alpha+e_{i}-e_{j})$.

Next we show that $v\in\mathcal{S}_{j,q}(\alpha+e_{i}-e_{j})$. Note that $\mathcal{S}_{j,q}(\alpha+e_{i}-e_{j})=\mathcal{S}_{j,q}(\alpha-e_{j})$
in this case. Put $\beta=\alpha-e_{j}$. Then applying Lemma \ref{lem:Gel-toda-2},
we have 
\begin{align*}
 & (x_{j}-x_{q})^{2}M_{j,q}(\alpha-e_{j})L_{i,j}(\alpha)=(x_{j}-x_{q})^{2}M_{j,q}(\beta)L_{i,j}(\alpha)\\
 & =\left\{ (\beta_{j}+1)\beta_{q}-L_{q,j}(\beta+e_{j}-e_{q})L_{j,q}(\beta)\right\} L_{i,j}(\alpha)\\
 & =\alpha_{j}\alpha_{q}L_{i,j}(\alpha)-L_{q,j}(\beta+e_{j}-e_{q})\left\{ L_{j,q}(\beta)L_{i,j}(\alpha)\right\} .
\end{align*}
Noting that $L_{j,q}(\beta)=L_{j,q}(\alpha)$ and $L_{q,j}(\beta+e_{j}-e_{q})=L_{q,j}(\alpha)$,
and applying (\ref{eq:gel-toda-7}) of Lemma \ref{lem:Gel-toda-6},
the second term of the last line above is written as 
\begin{align*}
L_{q,j}(\alpha)\{ L_{j,q}(\alpha)L_{i,j}(\alpha) \}  & \equiv\alpha_{j}L_{q,j}(\alpha)L_{i,q}(\alpha)\;\text{modulo}\;\mathcal{R}\cdot M_{q,j}(\alpha)\\
 & \equiv\alpha_{j}\alpha_{q}L_{i,j}(\alpha)\;\text{modulo}\;\mathcal{R}\cdot M_{i,j}(\alpha).
\end{align*}
Thus we have 
\[
(x_{j}-x_{q})^{2}M_{j,q}(\alpha-e_{j})L_{i,j}(\alpha)\equiv\alpha_{j}\alpha_{q}L_{i,j}(\alpha)-\alpha_{j}\alpha_{q}L_{i,j}(\alpha)=0
\]
modulo $\mathcal{R}\cdot M_{i,j}(\alpha)+\mathcal{R}\cdot M_{q,j}(\alpha)$. It follows that
\[
(x_{j}-x_{q})^{2}M_{j,q}(\alpha+e_{i}-e_{j})v=(x_{j}-x_{q})^{2}M_{j,q}(\alpha-e_{j})L_{i,j}(\alpha)u=0
\]
since $u\in\mathcal{S}_{i,j}(\alpha)\cap\mathcal{S}_{i,q}(\alpha)\cap\mathcal{S}_{j,q}(\alpha)$ and hence
$M_{i,j}(\alpha)u=M_{q,j}(\alpha)u=0$ holds. This proves $v\in\mathcal{S}_{j,q}(\alpha+e_{i}-e_{j})$.
\end{proof}

\subsection{Hypergeometric solution to the 2dTHE}

Now we can construct a solution of 2dTHE expressed in terms of the
Gelfand HGF. Consider the sequence $\{\mathcal{M}_{n}(\alpha)\}_{n\in\mathbb{Z}}$ of
the EPD equations:
\[
\mathcal{M}_{n}(\alpha):M_{p,q}(\alpha+n(e_{i}-e_{j}))u=0,\quad1\leq p\neq q\leq N.
\]
For the sake of brevity, we denote $M_{p,q}(\alpha+n(e_{i}-e_{j}))$ as
$M_{n;p,q}(\alpha)$. The set of holomorphic solutions of the system $\mathcal{M}_{n}(\alpha)$
is $\mathcal{S}(\alpha+n(e_{i}-e_{j}))$. Proposition \ref{prop:Gel-toda-1} says
that the operators $L_{i,j}(\cdot),L_{j,i}(\cdot)$ induce the map
\begin{align*}
H_{n} & :\mathcal{S}(\alpha+n(e_{i}-e_{j}))\to\mathcal{S}(\alpha+(n+1)(e_{i}-e_{j})),\\
B_{n} & :\mathcal{S}(\alpha+n(e_{i}-e_{j}))\to\mathcal{S}(\alpha+(n-1)(e_{i}-e_{j}))
\end{align*}
satisfying $B_{n+1}H_{n}=1,\;H_{n-1}B_{n}=1$ on $\mathcal{S}(\alpha+n(e_{i}-e_{j}))$,
where
\begin{align*}
H_{n} & =L_{i,j}(\alpha+n(e_{i}-e_{j}))= (x_{i}-x_{j})\partial_{j}+\alpha_{j}-n ,\\
B_{n} & =\frac{1}{(\alpha_{i}+n)(\alpha_{j}-n+1)}L_{j,i}(\alpha+n(e_{i}-e_{j}))\\
 & =\frac{1}{(\alpha_{i}+n)(\alpha_{j}-n+1)}\left\{ (x_{j}-x_{i})\partial_{i}+\alpha_{i}+n\right\} .
\end{align*}
We know that, for the EPD operator
\[
M_{n;i,j}(\alpha)=\partial_{i}\partial_{j}+\frac{\alpha_{j}-n}{x_{i}-x_{j}}\partial_{i}+\frac{\alpha_{i}+n}{x_{j}-x_{i}}\partial_{j},
\]
its normal form in the sense of Lemma \ref{lem:norm-1} is given by
\[
M_{n;i,j}'(\alpha)=\partial_{i}\partial_{j}+\frac{\alpha_{j}-\alpha_{i}-2n}{x_{i}-x_{j}}\partial_{i}+\frac{(\alpha_{i}+n)(\alpha_{j}-n+1)}{(x_{i}-x_{j})^{2}}.
\]
 Recall that the normal form $M_{n;i,j}'(\alpha)$ is obtained from $M_{n;i,j}(\alpha)$
as
\[
M_{n;i,j}'(\alpha)=(\mathrm{Ad}\, g_{n})M_{n;i,j}(\alpha):=g_{n}\cdot M_{n;i,j}(\alpha)\cdot g_{n}^{-1}
\]
with $g_{n}(x)=(x_{i}-x_{j})^{-(\alpha_{i}+n)}$. Thus we have the diagram
\begin{equation}
\begin{CD}M_{n+1;i,j}(\alpha)@>\mathrm{Ad}\, g_{n+1}>>M_{n+1;i,j}'(\alpha)\\
@AH_{n}AA@AAH_{n}'A\\
M_{n;i,j}(\alpha)@>\mathrm{Ad}\, g_{n}>>M_{n;i,j}'(\alpha)\\
@VB_{n}VV@VVB_{n}'V\\
M_{n-1;i,j}(\alpha)@>\mathrm{Ad}\, g_{n-1}>>M_{n-1;i,j}'(\alpha)
\end{CD}\label{eq:gel-toda-5}
\end{equation}
where the vertical arrow $H_{n}$ implies that the operator $M_{n+1;i,j}(\alpha)$
is determined from $M_{n;i,j}(\alpha)$ by the change of unknown $u\mapsto u'=L_{i,j}(\alpha+n(e_{i}-e_{j}))u$
for $M_{n;i,j}(\alpha)u=0$. In this situation, we can determine the operator
$H_{n}'$ so that the above diagram is commutative. We can show that
$H_{n}'$ is determined as 
\[
H_{n}'=\partial_{j}+\frac{\alpha_{j}-\alpha_{i}-2n}{x_{i}-x_{j}}.
\]
In fact, take a solution $v_{n}$ of $M_{n;i,j}'(\alpha)v=0$, then $u_{n}:=g_{n}^{-1}v_{n}$
is a solution of $M_{n;i,j}(\alpha)u=0$. Put $u_{n+1}=H_{n}u_{n}$ and
$v_{n+1}:=g_{n+1}u_{n+1}$. Then we see that $M_{n+1;i,j}'(\alpha)v_{n+1}=0$.
If the diagram (\ref{eq:gel-toda-5}) is commutative, $v_{n+1}$ should
be obtained as $v_{n+1}=H_{n}'v_{n}$. Since 
\[
v_{n+1}=g_{n+1}u_{n+1}=g_{n+1}H_{n}u_{n}=(g_{n+1}\cdot H_{n}\cdot g_{n}^{-1})v_{n},
\]
we should have
\begin{align*}
H_{n}' & =g_{n+1}\cdot H_{n}\cdot g_{n}^{-1}.\\
 & =(x_{i}-x_{j})^{-(\alpha_{i}+n+1)}\left\{ (x_{i}-x_{j})\partial_{j}+\alpha_{j}-n\right\} (x_{i}-x_{j})^{\alpha_{i}+n}\\
 & =(x_{i}-x_{j})^{-(\alpha_{i}+n)}\cdot\partial_{j}\cdot(x_{i}-x_{j})^{\alpha_{i}+n}+\frac{\alpha_{j}-n}{x_i -x_j}\\
 & =\partial_{j}+\frac{\alpha_{j}-\alpha_{i}-2n}{x_{i}-x_{j}}.
\end{align*}
This is just the contiguity operator (\ref{eq:darboux-1}) discussed
in Section \ref{subsec:Darboux-transformation}. Similarly, we can
determine $B_{n}'$ as 
\[
B_{n}'=g_{n-1}\cdot B_{n}\cdot g_{n}^{-1}=-\frac{(x_{i}-x_{j})^{2}}{(\alpha_{i}+n)(\alpha_{j}-n+1)}\partial_{i},
\]
which is just the contiguity operator (\ref{eq:darboux-1}) for the
Laplace sequence $\{M'_{n;i,j}(\alpha)\}$.

For a given $u_{0}(x)\in\mathcal{S}(\alpha)$, we define $\{u_{n}(x)\}_{n\in\mathbb{Z}}$,
$u_{n}\in\mathcal{S}(\alpha+n(e_{i}-e_{j}))$, by $u_{n+1}=H_{n}u_{n}\;(n\geq0)$
and $u_{n-1}=B_{n}u_{n}\;(n\leq0)$. Putting $u_{n}'(x):=g_{n}(x)u_{n}(x)$
with $g_{n}(x)=(x_{i}-x_{j})^{-(\alpha_{i}+n)}$, we have $M_{n;i,j}'(\alpha)u_{n}'=0$
for the Laplace sequence $\{M_{n;i,j}'(\alpha)\}_{n\in\mathbb{Z}}$ such that
$u_{n+1}'=H_{n}'u_{n}'$ and $u_{n-1}'=B_{n}'u_{n}'$ for all $n\in\mathbb{Z}$.
To obtain a solution to the 2dTHE
\begin{equation}
\partial_{i}\partial_{j}\log\tau_{n}=\frac{\tau_{n+1}\tau_{n-1}}{\tau^{2}},\quad n\in\mathbb{Z},\label{eq:gel-toda-9}
\end{equation}
we apply Proposition \ref{prop:Backlund-2} with the seed solution
obtained in Proposition \ref{prop:Laplace-3}. Here the seed solution
is $t_{n}=t_{n}(\alpha_{i},\alpha_{j};x_{i},x_{j})$, where
\[
t_{n}(\alpha,\beta;x,y)=B(\alpha,\beta;n)(x-y)^{p(\alpha,\beta;n)}
\]
with 
\begin{align*}
p(\alpha,\beta;n) & =(\alpha+n)(\beta-n+1),\\
B(\alpha,\beta;n) & =\begin{cases}
A^{n}\prod_{k=0}^{n-1}\left(\prod_{l=1}^{k}p(\alpha_{i},\alpha_{j};l)\right), & n\geq2,\\
A^{n}\prod_{k=1}^{|n|}\left(\prod_{l=-k+1}^{0}p(\alpha_{i},\alpha_{j};l)\right), & n\leq-1,
\end{cases}
\end{align*}
$B(\alpha,\beta;0)=1,B(\alpha,\beta;1)=A$, $A$ being an arbitrary constant. Then
we obtain the solution $\{\tau_{n}\}_{n\in\mathbb{Z}}$ to the 2dTHE (\ref{eq:gel-toda-9})
given by $\tau_{n}(x)=t_{n}(\alpha_{i},\alpha_{j};x_{i},x_{j})(x_{i}-x_{j})^{-(\alpha_{i}+n)}u_{n}(x)$.

In the above setting, as a particular case, we can take $u_{0}(x)$
as $u_{0}(x)=\Phi(x;\alpha)$ which is the Gelfand HGF $F(\mathbf{x};\alpha)$, see
(\ref{eq:red-0}). Then we can show 
\[
u_{n}(x)=\frac{\Gamma(\alpha_{j}+1)}{\Gamma(\alpha_{j}-n+1)}\Phi(x;\alpha+n(e_{i}-e_{j}))
\]
by using the contiguity relation (\ref{eq:2-red-4}) for $\Phi(x;\alpha)$.
Summarizing the above argument, we have following result.
\begin{thm}
\label{thm:main}We fix any pair $(i,j)$ such that $1\leq i\ne j\leq N$. 

(1) Take any $u_{0}(x)\in\mathcal{S}(\alpha)$ and define the sequence $\{u_{n}(x)\}_{n\in\mathbb{Z}}$
such that $u_{n}(x)\in\mathcal{S}(\alpha+n(e_{i}-e_{j}))$ by 
\begin{align*}
u_{n+1} & =H_{n}u_{n}\quad(n\geq0),\quad u_{n-1}=B_{n}u_{n}\quad(n\leq0),
\end{align*}
where
\begin{align*}
H_{n} & =L_{i,j}(\alpha+n(e_{i}-e_{j}))=\left\{ (x_{i}-x_{j})\partial_{j}+\alpha_{j}-n\right\} ,\\
B_{n} & =\frac{1}{(\alpha_{i}+n)(\alpha_{j}-n+1)}L_{j,i}(\alpha+n(e_{i}-e_{j}))\\
 & =\frac{1}{(\alpha_{i}+n)(\alpha_{j}-n+1)}\left\{ (x_{j}-x_{i})\partial_{i}+\alpha_{i}+n\right\} .
\end{align*}
Then $\tau_{n}(x)=B(\alpha_{i},\alpha_{j};n)\cdot(x_{i}-x_{j})^{(\alpha_{i}+n)(\alpha_{j}-n)}u_{n}(x)$
gives a solution of the 2dTHE, where $B(0)=1,B(1)=A$ and
\[
B(\alpha,\beta;n)=\begin{cases}
A^{n}\prod_{k=0}^{n-1}\left(\prod_{l=1}^{k}p(\alpha,\beta;l)\right), & n\geq2,\\
A^{n}\prod_{k=1}^{|n|}\left(\prod_{l=-k+1}^{0}p(\alpha,\beta;l)\right), & n\leq-1.
\end{cases}
\]
with an arbitrary constant $A$.

(2) Let $\Phi(x;\alpha)$ be the Gelfand HGF defined by $\Phi(x;\alpha)=\int_{C}\prod_{k=1}^{N}(u+x_{k})^{\alpha_{k}}du$.
Then 
\[
\tau_{n}(x)=\frac{\Gamma(\alpha_{j}+1)}{\Gamma(\alpha_{j}-n+1)}B(\alpha_{i},\alpha_{j};n)\cdot(x_{i}-x_{j})^{(\alpha_{i}+n)(\alpha_{j}-n)}\Phi(x;\alpha+n(e_{i}-e_{j}))
\]
gives a solution of the 2dTHE (\ref{eq:gel-toda-9}). 
\end{thm}


\begin{thebibliography}{10}
\bibitem{Appell-2} Appell, P.; Kamp\'e de F\'eriet, J., Fonction
hyperg\'eom\'etrique et hypersph\'eriques, Gauthier Villars, Paris,
(1926).

\bibitem{Darboux} Darboux, G., Lecon sur la Th\'eorie G\'en\'eral
des Surfaces, t.II, t.IV, Chelsea, New York, (1972).

\bibitem{Erdelyi} Erdelyi, A. et al., Higher transcendental functions,
vol I,II, R. E. Krieger Pub. Co (1981).

\bibitem{Gelfand} Gelfand, I. M., General theory of hypergeometric
functions, Soviet Math. Dokl. 33 (1986), 9--13.

\bibitem{Hirota} Hirota, R.; Ohta, Y.; Satsuma, J., Wronskian structures
of solutions for soliton equations, Progr. Theoret. Phys. Suppl. No.
94 (1988), 59--72.

\bibitem{Horikawa} Horikawa, E., Transformations and contiguity relations
for Gelfand's hypergeometric functions, J. Math. Sci. Univ. Tokyo,
1 (1994), 181-203.

\bibitem{IKSY} Iwasaki, K.; Kimura, H.; Shimomura, S. ; Yoshida,
M., From Gauss to Painlev\'e, Vieweg Verlag, (1991).

\bibitem{kame-1} Kametaka, Y., Hypergeometric solutions of Toda equation,
S\^urikaisekikenky\^usho K\^okyuroku 554 (1985), 26-46.

\bibitem{Kimura-Haraoka}H. Kimura, Y. Haraoka and K. Takano, The
generalized confluent hypergeometric functions, Proc. Japan Acad.
68 (1992), 290--295.

\bibitem{Kimura-H-T} Kimura, H.; Haraoka, Y.; Takano, K., On contiguity
relations of the confluent hypergeometric systems, Proc. Japan Acad.
Ser. A Math. Sci. 70 (1994), no. 2, 47--49.

\bibitem{Nakamura-A} Nakamura, A., Toda equation and its solutions
in special functions. J. Phys. Soc. Japan 65 (1996), no. 6, 1589--1597.

\bibitem{Okamoto-1} Okamoto, K., Sur les \'echelles associ\'ees
aux fonctions sp\'eciales et l'\'equation de Toda, J. Fac. Sci.
Univ. Tokyo Sect. IA Math. 34 (1987), no. 3, 709--740.

\bibitem{Okamoto-2} Okamoto, K., \'Echelles et l'\'equation de
Toda, Publ. Inst. Rech. Math. Av., Universit\'e Louis Pasteur, 1988,
19--38.

\bibitem{Olshanetsky} Olshanetsky, M.A.; Perelomov, A.M., Explicit
solutions of classical generalized Toda models, Invent. Math. 54 (1979),
261-269.

\bibitem{sasaki} Sasaki, T., Contiguity relations of Aomoto-Gelfand
hypergeometric functions and applications to Appell's system $F_{3}$
and Goursat's $_{3}F_{2}$, SIAM J. Math. Anal. 22 (1991), 821--846.

\bibitem{Satuma} Tokihiro, T.; Satsuma, J.; Willox, R., On special
function solutions to nonlinear integrable equations. Physics Letters
A 236 (1997) 21-29.

\end{thebibliography}
\end{document}